\documentclass{amsart}%
\usepackage{amsmath, amsthm}
\usepackage{amsfonts,amssymb}
\usepackage{graphicx}%
\usepackage{amsmath}%
\usepackage{amsfonts}%
\usepackage{amssymb}

\newtheorem{theorem}{Theorem}[section]

\newtheorem{prop}[theorem]{Proposition}

\theoremstyle{definition}
\newtheorem{definition}[theorem]{Definition}
\newtheorem{example}[theorem]{Example}

\numberwithin{equation}{section}
\email{michael.barnsley@anu.edu.au}
\email{avince@ufl.edu}
\begin{document}
\title[Fractal Tiling]{Fractal Tiling}
\author[M. Barnsley]{Michael Barnsley}
\address{The Australian National University \\
Canberra, Australia}
\author[A. Vince]{Andrew Vince}
\address{Department of Mathematics, Universityof Florida \\
Gainesville, FL, USA }
\maketitle

\begin{abstract}
A simple, yet unifying method is provided for the construction of tilings by
tiles obtained from the attractor of an iterated function system (IFS). Many
examples appearing in the literature in ad hoc ways, as well as new examples,
can be constructed by this method. These tilings can be used to extend a
fractal transformation defined on the attractor of a contractive IFS to a
fractal transformation on the entire space upon which the IFS acts.

\end{abstract}

\section{Introduction}

\label{sec:Intro} The subject of this paper is fractal tilings of
${\mathbb{R}}^{n}$ and, more generally, complete metric spaces.  By
``fractal", we mean that each tile is the attractor of an iterated function
system (IFS). Computer generated drawings of tilings of the plane by
self-similar fractal figures appear in papers beginning in the 1980's and
1990's, for example the lattice tiling of the plane by copies of the
twindragon. Tilings constructed from an IFS often possess global symmetry and
self-replicating properties. Research on such tilings include the work of
numerous people including Akiyama, Bandt, Gelbrich, Gr\" ochenig, Hass,
Kenyon, Lagarias, Lau, Madych, Radin, Solomyak, Strichartz, Thurston, Vince,
and Wang; see for example \cite{aki, bandt2,Ge,gh,gm, kenyon, lau, lw,radin,
solomyak, str,thurstonG,V2, V1} and the references therein. Even aperiodic
tilings can be put into the IFS context; see for example the fractal version
of the Penrose tilings \cite{bg}.

The use of the inverses of the IFS functions in the study of tilings is
well-established, in particular in many of the references cited above. The
main contribution of this paper is a simple, yet unifying, method for
constructing tilings from an IFS. Many examples appearing in the literature in
ad hoc ways, as well as new examples, can be constructed by this method. 
The usefulness of the method is demonstrated by showing by applying it to the
construction of fractal transformations between basins of attractors of pairs of IFSs.

Section~\ref{sec:IFS} contains background on iterated function systems, their
attractors, and addresses. Tiling are constructed for both non-overlapping and
overlapping IFSs, concepts defined in Section~\ref{sec:non-overlapping}. The
tilings constructed in the non-overlapping case, in
Section~\ref{sec:non-overlapping}, include well known examples such as the
digit tilings, crystallographic tilings, and non-periodic tilings by copies of
a polygon, as well as new tilings in both the Euclidean and real projective
planes. The tilings constructed in the overlapping case, in
Section~\ref{sec:O}, depend on the choice of what is called a mask and are
new. Theorems~\ref{thm:basinN}, \ref{thm:random}, and \ref{thm:basinO} state
that, in both of the above cases, when the attractor has nonempty interior,
the resulting tiling very often covers the entire space. Tilings associated
with an attractor with  empty interior are also of interest and include
tilings of the graphs of fractal continuations of fractal functions, a
relatively new development in fractal geometry \cite{BV3}. The methods for
obtaining tilings from an IFS can be extended to a graph iterated function
system (GIFS); this is done in Section~\ref{sec:GIFS}. The Penrose tilings, as
well as new examples, can be obtained by this general construction.
Section~\ref{sec:FT} describes how, using our tiling method, a fractal
transformation, which is a special kind of mapping from the attractor of one
IFS to the attractor of another, can be extended from the attractor to the
entire space, for example, from ${\mathbb{R}}^{n}$ to ${\mathbb{R}}^{n}$ in
the Euclidean case.

\section{Iterated Function Systems, Attractors, and Addresses}

\label{sec:IFS}

Let $({\mathbb{X}},d_{{\mathbb{X}}})$ be a complete metric space and let $F$
be an IFS with attractor $A$; the definitions follow:

\begin{definition}
If $f_{n}:\mathbb{X}\rightarrow\mathbb{X}$, $n=1,2,\dots,N,$ are continuous
functions, then $F=\left(  \mathbb{X};f_{1},f_{2},...,f_{N}\right)  $ is
called an \textbf{iterated function system} (IFS). If each of the maps $f\in
F$ is a homeomorphism then $F$ is said to be \textbf{invertible}, and the
notation $F^{*}:=\left(  \mathbb{X};f_{1}^{-1},f_{2}^{-1},...,f_{N}%
^{-1}\right)$  is used.
\end{definition}

Subsequently in this paper we refer to some special cases. An IFS is called
\textbf{affine} if $\mathbb{X }= {\mathbb{R}}^{n}$ and the functions in the
IFS are affine functions of the form $f(x) = Ax+a$, where $A$ is an $n\times
n$ matrix and $a \in{\mathbb{R}}^{n}$. An IFS is called a \textbf{projective
IFS} if $\mathbb{X }= \mathbb{RP}^{n}$, real projective space, and the
functions in the IFS are projective functions of the form $f(x) = Ax$, where
$A$ is an $(n+1)\times(n+1)$ matrix and $x$ is given by homogeneous
coordinates.
\medskip

By a slight abuse of terminology we use the same symbol $F$ for the IFS, the
set of functions in the IFS, and for the following mapping. Letting
$2^{\mathbb{X}}$ denote the collection of subsets of $\mathbb{X}$, define
$F:2^{\mathbb{X}}\mathbb{\rightarrow}2^{\mathbb{X}}$ by
\[
F(B)=\bigcup_{f\in F}f(B)
\]
for all $B\in2^{\mathbb{X}}$. Let $\mathbb{H=H(X)}$ be the set of nonempty
compact subsets of $\mathbb{X}$. Since $F\left(  \mathbb{H}\right)
\subseteq\mathbb{H}$ we can also treat $F$ as a mapping
$F:\mathbb{H\rightarrow H}$. Let $d_{\mathbb{H}}$ denote the Hausdorff metric
on $\mathbb{H}$.

For $B\subset\mathbb{X}$ and $k\in\mathbb{N}:=\{1,2,...\}$, let $F ^{k}(S)$
denote the $k$-fold composition of $F$, the union of $f_{i_{1}}\circ f_{i_{2}%
}\circ\cdots\circ f_{i_{k}}(S)$ over all finite words $i_{1}i_{2}\cdots i_{k}$
of length $k.$ Define $F^{0}(S)=S.$

\begin{definition}
\label{def:attractor} A non-empty set $A\in\mathbb{H(X)}$ is said to be an
\textbf{attractor} of the IFS $F$ if

(i) $F(A)=A$ and

(ii) there is an open set $U\subset\mathbb{X}$ such that $A\subset U$ and
$\lim_{k\rightarrow\infty}F^{k}(S)=A$ for all $S\in\mathbb{H(}U)$, where the
limit is with respect to the Hausdorff metric.

The largest open set $U$ such that (ii) is true, i.e. the union of all open
sets $U$ such that (ii) is true, is called the \textbf{basin} of the attractor
$A$ with respect to the IFS $F$ and is denoted by $B=B(A)$.
\end{definition}

An IFS $F$ on a metric space $(\mathbb{X},d)$ is said to be
\textbf{contractive} if there is a metric $\hat{d}$ inducing the same topology
on $\mathbb{X}$ as the metric $d$ with respect to which the functions in $F$
are strict contractions, i.e., there exists $\lambda\in\lbrack0,1)$ such that
$\hat{d}_{\mathbb{X}}(f(x),f(y))\leq\lambda\hat{d}_{\mathbb{X}}(x,y)$ for 
all $x,y\in\mathbb{X}$ and for all $f\in F$. A classical result of Hutchinson
\cite{H} states that if $F$ is contractive on a complete metric space
$\mathbb{X}$, then $F$ has a unique attractor with basin $\mathbb{X}$.

Let $[N]=\{1,2,\dots,N\}$, with $[N]^{k}$ denoting words of length $k$ in the
alphabet $[N],$ and $[N]^{\infty}$ denoting infinite words of the form
$\omega_{1}\omega_{2}\cdots$, where $\omega_{i}\in\lbrack N]$ for all $i$.
\textquotedblleft Word" will always refer to an infinite word unless otherwise
specified. A \textbf{subword} (either finite or infinite) of a word $\omega$
is a string of consecutive elements of $\omega$. The length $k$ of $\omega
\in\lbrack N]^{k}$ is denoted by $|\omega|=k$. For $\omega\in\lbrack N]^{k}$
we use the notation $\overline{\omega}=\omega\omega\omega\cdots.$. For
$\omega\in\lbrack N]^{\infty}$ and $k\in\mathbb{N}:=\{1,2,..\}$ the notation
\[
\omega|k:=\omega_{1}\omega_{2}\cdots\omega_{k}%
\]
\vskip2mm

\noindent is introduced, and for an IFS $F=\left(  \mathbb{X};f_{1}%
,f_{2},...,f_{N}\right) $ and $\omega\in[N]^{k}$, the following shorthand
notation will be used for compositions of function:%

\[
\begin{aligned} f_{\omega} &:=  f_{\omega_{1}}\circ f_{\omega_{2}} \circ\cdots \circ f_{\omega_{k}}, \\
( f^{-1})_{\omega} &:= f^{-1} _{\omega_{1}}\circ f^{-1} _{\omega_{2}} \circ\cdots \circ f^{-1} _{\omega_{k}}
\end{aligned}
\]
\medskip

\noindent Note that, in general, $(f^{-1})_{\omega} \neq(f_{\omega})^{-1}$.
\medskip

In order to assign \textit{addresses} to the points of an attractor of an IFS,
it is convenient to introduce the \textit{point-fibred} property. The
following notions concerning point-fibred attractors \cite{BV4} derive from
results of Kieninger \cite{kieninger}, adapted to the present setting. Let $A$
be an attractor of $F$ on a complete metric space ${\mathbb{X}}$, and let $B$
be the basin of $A$. The attractor $A$ is said to be \textbf{point-fibred}
with respect to $F$ if (i) $\left\{  f_{\omega|k}(C)\right\}  _{k=1}^{\infty}$
is a convergent sequence in $\mathbb{H(}{\mathbb{X)}}$ for all $C\in
\mathbb{H(X)}$ with $C\subset B(A)$; (ii) the limit is a singleton whose value
is independent of $C$. If $A$ is point-fibred (w.r.t. $F)$ then we denote the
limit in (ii) by $\left\{  \pi(\omega)\right\}  \subset A$.

\begin{definition}
If $A$ is a point-fibred attractor of $F$, then, with respect to $F$ and $A$,
the \textbf{coordinate map} $\pi\,:\,[N]^{\infty}\rightarrow A$ is defined by
\begin{equation}
\pi(\omega)=\lim_{k\rightarrow\infty}f_{\omega|k}(x_{0}). \label{coordmapeq}%
\end{equation}
The limit is independent of $x_{0} \in B(A)$. For $x\in A$, any word in
$\pi^{-1}(x)$ is called an \textbf{address} of $x$.
\end{definition}

Equipping $[N]^{\infty}$ with the product topology, the map $\pi:[N]^{\infty
}\rightarrow A$ is continuous and onto. If $F$ is contractive, then $A$ is
point-fibred (w.r.t. $F$). A point belonging to the attractor of a
point-fibred IFS has at least one address.

\section{Fractal Tilings from a non-overlapping IFS}

\label{sec:non-overlapping}

Let $S^{o}$ denote the interior of a set $S$ in a complete metric space
$\mathbb{X}$. Two sets $X$ and $Y$ are \textit{overlapping} if $(X\cap Y)^{o}
\neq\emptyset$.

\begin{definition}
In a metric space $\mathbb{X}$, a \textbf{tile} is a nonempty compact set. A
\textbf{tiling} of a set $S$ is a set of non-overlapping tiles whose union is
$S$.
\end{definition}

\begin{definition}
An attractor $A$ of an IFS $F=({\mathbb{X}};f_{1},f_{2},\dots,f_{N})$ is
called \textbf{overlapping} (w.r.t. $F$) if $f(A)$ and $g(A)$ are overlapping
for some $f,g\in F$. Otherwise $A$ is called \textbf{non-overlapping (}w.r.t.
$F$). A non-overlapping attractor $A$ is either \textbf{totally disconnected},
i.e, $f(A)\cap g(A)=\emptyset$ for all $f,g\in F$ or else$\ A$ is called
\textbf{just touching}. Note that a non-overlapping attractor with non-empty
interior must be just touching.
\end{definition}

Starting with an attractor $A$ of an invertible IFS
\[
F=\{X\,;\,f_{1},f_{2},\dots,f_{N}\},
\]
potentially an infinite number of tilings can be constructed from $A$ and $F$
- one tiling $T_{\theta}$ for each word $\theta\in\lbrack N]^{\infty}$. Given
a word $\theta\in\lbrack N]^{\infty}$ and a positive integer $k$, for any
$\omega\in\lbrack N]^{k}$, let%

\[
\begin{aligned} t_{\theta,\omega} &= ((f^{-1})_{\theta|k}\circ f_{\omega})(A), \\
T_{\theta, k} &= \{ t_{\theta,\omega} \, : \, \omega \in [N]^k\}.
\end{aligned}
\]
Since the IFS $F$ is non-overlapping, the sets in $T_{\theta, k}$ do not
overlap. Since, for any $\omega\in\lbrack N]^{k}$, we have
\[
(f^{-1})_{\theta|{k}}\circ f_{\omega}=(f^{-1})_{\theta|{k}}\circ
(f_{\theta_{k+1}})^{-1}\circ f_{\theta_{k+1}}\circ f_{\omega}=(f^{-1}%
)_{\theta|{k+1}} \circ f_{\theta_{k+1}\omega},
\]
the inclusion
\[
T_{\theta}\subset T_{\theta,k+1}%
\]
holds for all $k$. Therefore%
\begin{equation}
\label{eq:T}T_{\theta} :=\bigcup_{k=1}^{\infty}T_{\theta,k}%
\end{equation}
is a tiling of
\[
B(\theta):=\bigcup_{k=1}^{\infty}(f^{-1})_{\theta|{k}}(A)\subseteq{X},
\]
where the union is a nested union. Note that $B(\theta)$ is connected if $A$
is connected, and $B(\theta)$ has non-empty interior if $A$ has non-empty interior.

If $F$ is contractive and the unique attractor $A$ has non-empty interior, it
is possible to tile the entire space ${\mathbb{X}}$ with non-overlapping
copies of $A$. More specifically, if $F$ is contractive and $\theta$ satisfies
a not too restrictive condition given below, then $T_{\theta}$ tiles the
entire space ${\mathbb{X}}$. In the case that $T_{\theta}$ tiles ${\mathbb{X}%
}$, we call $T_{\theta}$ a \textbf{full tiling}.

\begin{example}
The IFS $F=\{{\mathbb{R}};\,f_{1},f_{2}\}$ where $f_{1}(x)=x/2$ and
$f_{2}(x)=x/2+1/2$ has attractor $[0,1]$. In this case, the tiling
$T_{\overline{1}}$ is the tiling of $[0,\infty)$ by unit intervals. The
tiling $T_{\overline{2}}$ is the tiling of $(-\infty,1]$ by unit intervals.
If $\theta\neq{\overline{1}},{\overline{2}}$, then $T_{\theta}$ is the full
tiling of ${\mathbb{R}}$ by unit intervals.
\end{example}

\begin{definition}
\label{def:full} For an IFS $F$ with attractor $A$, call $\theta\in
[N]^{\infty}$ \textbf{full} if there exists a nonempty compact set $A^{\prime
o}$ such that, for any positive integer $M$, there exist $n>m\geq M$ such
that
\[
f_{\theta_{n}}\circ f_{\theta_{n-1}} \circ\cdots\circ f_{\theta_{m+1}} (A)
\subset A^{\prime}.
\]

\end{definition}

By Proposition~\ref{prop:reversible} and Theorem~\ref{thm:random}, nearly all
words $\theta$ are full.

\begin{definition}
\label{def:reversible} Call $\theta\in\lbrack N]^{\infty}$ \textbf{reversible}
w.r.t. a point-fibred attractor $A$ and IFS $F$ if there exists an
$\omega=\omega_{1}\omega_{2}\cdots\in\lbrack N]^{\infty}$ such that $\omega$
is the address of some point in $A^{o}$ and, for every pair of positive
integer $M$ and $L$, there is an integer $m\geq M$ such that%
\begin{equation}
\omega_{1}\omega_{2}\cdots\omega_{L} = \theta_{m+L}\theta_{m+L-1}\cdots
\theta_{m+1}. \label{reveqn}%
\end{equation}
Call $\theta\in\lbrack N]^{\infty}$ \textbf{strongly reversible } w.r.t. $A$
and $F$ if there exists $\omega=\omega_{1}\omega_{2}\cdots\in\lbrack
N]^{\infty}$ such that $\omega$ is the address of some point in $A^{o}$ and,
for every positive integer $M$, there is an integer $m\geq M$ such that%
\begin{equation}
\omega_{1}\omega_{2}\cdots\omega_{m} = \theta_{m}\theta_{m-1}\cdots\theta_{1}.
\label{sreveqn}%
\end{equation}

\end{definition}

\medskip

By statement (3) of Theorem~\ref{prop:reversible} below, if $\theta$ is
strongly reversible,  then $\theta$ is reversible. The converse, however, is
not true; for example, the string  $\theta=3\overline{12}$ is reversible but
not strongly reversible. To see that $3\overline{12}$ is reversible, choose
$\omega=\overline{12}$.

Let $S:[N]^{\infty}\rightarrow\lbrack N]^{\infty}$ denote the usual shift map
and introduce the notation $\overleftarrow{\omega|m}:=\omega_{m}\omega
_{m-1}\cdots\omega_{2}\omega_{1}$. Equation \ref{reveqn} can be expressed
$S^{m}\theta|L=\overleftarrow{\omega|L}$, while equation \ref{sreveqn} can be
expressed $\theta|m=\overleftarrow{\omega|m}$. Note that Definition
\ref{def:reversible} is equivalent to: there is a point $\omega$, whose
image under $\pi_{F}$ lies in the interior of $A$, such that, if it is
truncated to any given finite length and then reversed, then it occurs as a
finite subword of $\theta$ infinitely many times.

\begin{definition}
A word $\theta\in[N]^{\infty}$ is \textbf{disjunctive} if every finite word is
a subword of $\theta$. In fact, if $\theta$ is disjunctive, then every finite
word (in the alphabet $[N]$) appears as a subword in $\sigma$ infinitely many times.
\end{definition}

\begin{theorem}
\label{prop:reversible} For an IFS $F$, let $A$ be a point-fibred attractor.
With respect to $A$ and $F$:

\begin{enumerate}
\item There are infinitely many disjunctive words in $[N]^{\infty}$ for
$N\geq2$.

\item If $A^{o} \neq\emptyset$, then every disjunctive word is strongly reversible.

\item Every strongly reversible word is reversible.

\item A word is reversible if and only if it is full.
\end{enumerate}
\end{theorem}

\begin{proof}
Not only are there infinitely many disjunctive words for $N\geq2$, but the set
of disjunctive sequences is a large subset of $[N]^{\infty}$ in a topological,
in a measure theoretic, and in an information theoretic sense \cite{S}.

Concerning statement (2), assume that $\theta$ is disjunctive. Let $a\in
A^{o}$ and let $\sigma$ be an address of $a$. There is a neighborhood $N(a)$
of $a$ that lies in $A$ and an integer $s$ such that any point whose address
starts with $\sigma|s$ lies in $N(A)\subset A$. Statement (2) will now be
proved by induction. We claim that there exists a word $\omega$ and an
increasing sequence $t_{n}\geq n$ of integers and such that the following
statements hold:

\begin{enumerate}
\item $\omega|{t_{n}}=\overleftarrow{\theta|{t_{n}}}$,

\item $\omega|{t_{n}}=\sigma|{s}\overleftarrow{\theta|{t_{n}-s}}$.
\end{enumerate}

\noindent The two conditions guarantee that $\theta$ is stongly reversible.
The case $n=1$ is shown as follows. Since $\theta$ is disjunctive, there is a
$t_{1}>s$ such that $\sigma|s=\theta_{t_{1}}\theta_{t_{1}-1}\cdots
\theta_{t_{1}-s+1}$. Define
\[
\omega|{t_{1}}:=\overleftarrow{\theta|{t_{1}}}=\sigma|{s}\overleftarrow{\theta
|{t_{1}-s}}.
\]
We next go from $n$ to $n+1$. Since $\theta$ is disjunctive, there is a
$t_{n+1}>t_{n}$ such that
\[
\omega|{t_{n}}=\theta_{t_{n+1}}\theta_{t_{n+1}-1}\cdots\theta_{t_{n+1}%
-t_{n}+1}.
\]
Now define
\[
\omega|{t_{n+1}}:=\overleftarrow{\theta|{t_{n+1}}}=\omega|{t_{n}%
}\overleftarrow{\theta|{t_{n+1}-t_{n}}}=\sigma|{s}\,\overleftarrow{\theta
|{t_{n+1}-s}}.
\]
The second equality guarantees that the initial $t_{n}$ elements of $\omega$
are unchanged; the third equality follows from condition (2) of the induction hypothesis.

Concerning statement (3), assume that $\theta$ is strongly reversible, and let
positive integers $M$ and $L$ be given. If $\omega$, as required, has the
strong reversal property with respect to $\theta$,  then there exists an
integer $n\geq M+L$ so that $\theta_{n}\theta_{n-1}\cdots\theta_{1}
=\omega_{1}\omega_{2}\cdots\omega_{n}$, from which it follows that there is an
integer $m \geq M$ such that  $\theta_{m+L}\theta_{m+L-1}\cdots\theta
_{1}=\omega_{1}\omega_{2}\cdots\omega_{m+L}$ and therefore $\theta_{m+L}%
\theta_{m+L-1}\cdots\theta_{m+1}=$ \linebreak$\omega_{1}\omega_{2}\cdots
\omega_{L}$. It follows that $\theta$ is reversible.

Concerning one direction of statement (4), assume that $\theta$ is reversible,
and let $\omega$ have the required reversal property w.r.t. $\theta$. Since
$\pi(\omega)\in A^{\circ}$ it follows that there is $L\in\mathbb{N}$ so that
$f_{\omega|L}(A)\subset A^{\circ}$. Choose $A^{\prime}=f_{\omega|L}(A)$ in
Definition~\ref{def:full} of full. Let $M$ be given. By the definition of
$\theta$ reversible, there exists an $m\geq M$ such that
\[
f_{\theta_{m+L}}\circ f_{\theta_{m+L-1}}\circ\cdots\circ f_{\theta_{m+1}%
}(A)=f_{\omega_{1}}\circ f_{\omega_{2}}\circ\cdots\circ f_{\omega_{L}%
}(A)=A^{\prime}.
\]
Taking $n = m+L$ in Definition~\ref{def:full}, it follows that $\theta$ is full.

Concerning the other direction of statement (4), suppose that $\theta$ is full
and let $A^{\prime}$ be the corresponding compact set. It follows that, for
any positive integer $M$, there is $n>m\geq M$ so that $f_{\theta_{n}}\circ
f_{\theta_{n-1}}\cdots\circ f_{\theta_{m+1}}(A)\subset A^{\prime}\subset
A^{\circ}$, which implies that $f_{\theta_{n_{k}}}\circ f_{\theta_{n_{k}-1}%
}\cdots\circ f_{\theta_{1}}(A)\subset A^{\prime}\subset A^{\circ}$ for an
infinite strictly increasing sequence $\left\{  n_{k}\right\}  _{k=1}^{\infty
}$ of positive integers. The set
\[
\lbrack N]^{\infty}\cup\bigcup_{k=1}^{\infty}[N]^{k}%
\]
becomes a compact metric space when endowed with an appropriate metric (words
with a long common prefix are close). Therefore the sequence of finite words
$\left\{  \theta_{n_{k}}\theta_{n_{k}-1}...\theta_{1}\right\}  _{k=1}^{\infty
}$ has a convergent subsequence, which we continue to denote by the same
notation $\left\{  \theta_{n_{k}}\theta_{n_{k}-1}...\theta_{1}\right\}
_{k=1}^{\infty}$, with limit $\omega\in\lbrack N]^{\infty}.$ Hence, for any
$L$ and $k$ sufficiently large, $\theta_{n_{k}}\theta_{n_{k}-1}...\theta
_{n_{k}-L+1}=\omega_{1}\omega_{2}...\omega_{L}$. It follows that $\theta$ is
reversible with reverse word $\omega$.
\end{proof}

\begin{theorem}
\label{thm:basinN} For an invertible IFS $F=\left(  \mathbb{X};f_{1}%
,f_{2},...,f_{N}\right) ,$ let $A$ be a just-touching attractor with non-empty
interior. If $\theta$ is a full word, then $T_{\theta}$ is a tiling of the set
$B(\theta)$ which contains the basin of $A$. If $F$ is contractive, then
$T_{\theta}$ is a full tiling of $\mathbb{X}$.
\end{theorem}

\begin{proof}
The proof is postponed because it is a special case of
Theorem~\ref{thm:basinO} in Section~\ref{sec:O}.
\end{proof}

By the above theorem, if $F$ is contractive and $\theta$ is full, then
$T_{\theta}$ tiles the entire space ${\mathbb{X}}$. According to
Proposition~\ref{prop:reversible}, full words are plentyful. According to the
next result, if $F$ is contractive, for \textit{any} infinite word $\theta$,
$T_{\theta}$ tiles ${\mathbb{X}}$ with probability $1$. Define a word
$\theta\in\lbrack N]^{\infty}$ to be a \textbf{random word} if there is a
$p>0$ such that each $\theta_{k},\,k=1,2,\dots$, is selected at random from
$\{1,2,...,N\}$ where the probability that $\theta_{k}=n,\,n\in\lbrack N],$ is
greater than\ or equal to $p,$ independent of the preceeding outcomes.

\begin{theorem}
\label{thm:random} Let $F = \{{\mathbb{X}} ;\, f_{1}, f_{2}, \dots, f_{N}\}$,
where ${\mathbb{X}}$ is compact, be a just touching invertible IFS with
attractor $A$ with non-empty interior. If $\theta\in[N]^{\infty}$ is a random
word, then, with probability $1$, the tiling $T_{\theta}$ covers the basin
$B(A)$. If ${\mathbb{X}}$ is contractive, then $T_{\theta}$ is a full tiling
of ${\mathbb{X}}$.
\end{theorem}

\begin{proof}
Let $x$ lie in the basin $B$ of the attractor $A$ of $F$, and let $\theta$ be
a random word.  Then $x$ lies in the union of the tiles of  $T_{\theta}$ if
and only if $x \in f^{-1}_{\theta|n}(A)$ for some $n$ if and only if
$f_{\theta_{n}}\circ f_{\theta_{n-1}} \circ\cdots\circ f_{\theta_{2}} \circ
f_{\theta_{1}}(x) \in A$ for some $n$. Given a word $\omega$, consider a
sequence $\left\{  x_{k}\right\}  _{k=0}^{\infty}$ of points in $\mathbb{X}$
defined by $x_{k}=f_{\omega_{k}}(x_{k-1}), \, k\geq1$. If $\omega$ is a random
word, then $\left\{  x_{k}\right\}  _{k=0}^{\infty}$ is called a
\textit{random orbit} of the point $x_{0}$. Since $\theta$ is assumed random,
if $y$ is any point in the interior of $A$ and $Q$ is an open neighborhood of
$y$ contained in $A$, then, according to \cite[Theorem 1]{BV2}, with
probability $1$ there is a point in the random orbit of any point $x$ in the
basin that lies in $Q \subset A$, i.e., with probability $1$, $f_{\theta_{n}%
}\circ f_{\theta_{n-1}} \circ\cdots\circ f_{\theta_{1}}(x) \in Q$ for some
$n$. Therefore, with probability $1$, there is a ball $B_{x}$ centered at $x$
such that
\[
f_{\theta_{n}}\circ f_{\theta_{n-1}} \circ\cdots\circ f_{\theta_{1}}(B_{x})
\subset Q \subset A.
\]

Next, for any $\epsilon>0$ let $\overline{B}_{\epsilon}$ be the open
$\epsilon$-neighborhood of the complement $\overline B$ of $B$, and let
$B_{\epsilon} = B\setminus\overline{B}_{\epsilon}$. Since $B_{\epsilon}$ is
compact, $B_{\epsilon}$ has a finite covering by balls $B_{x}$ and hence, with
probability $1$, there is an $n$ such that
\[
f_{\theta_{n}}\circ f_{\theta_{n-1}} \circ\cdots\circ f_{\theta_{1}%
}(B_{\epsilon}) \subset A.
\]
Therefore, for any $\epsilon>0$, with probability $1$, the tiling $T_{\theta}$
covers $B_{\epsilon}$. Assume that $T_{\theta}$ does not cover $B$. Then there
exists an $\epsilon$ such that $T_{\theta}$ does not cover $B_{\epsilon}$. But
the probability of that is $0$. If $F$ is contractive, the basin $B$ is
${\mathbb{X}}$.
\end{proof}

\begin{example}
(Digit tilings of ${\mathbb{R}}^{n}$) The terminology ``digit tiling'' comes
from the data used to construct the tiling, which is analogous to the usual
base and digits used to represent the integers. An \textit{expanding matrix}
is an $n\times n$ matrix such that the modulus of each eigenvalue is greater
than 1. Let $L$ be an $n\times n$ expanding integer matrix. A set $D =
\{d_{1}, d_{2}, \dots, d_{N}\}$ of coset representatives of the quotient
$\mathbb{Z}^{n}/L(\mathbb{Z}^{n})$ is called a \textit{digit set}. It is
assumed that $0 \in D$. By standard algebra results, for $D$ to be a digit set
it is necessary that
\[
|D| = |\det L \,|.
\]
Consider an affine IFS $F := F(L,D) = ({\mathbb{R}}^{n} \, ; \,f_{1}, f_{2},
\dots, f_{N})$, where
\[
f_{i}(x) = A^{-1}(x - d_{i}).
\]
Since $L$ is expanding, it is known that, with respect to a metric equivalent
to the Euclidean metric, $L^{-1}$ is a contraction. Since $F$ is contractive,
there is a unique attractor $A$ called a \textit{digit tile}. The basin of $A$
is all of ${\mathbb{R}}^{n}$. Note that a digit tile is completely determined
by the pair $(L,D)$ and will be denoted $T(L,D)$. It is known \cite{V1} that a
digit tile $T$ is the closure of its interior and its boundary has Lebesque
measure $0$. If $\theta\in[N]^{\infty}$ is full, then $T_{\theta}$ is a tiling
of ${\mathbb{R}}^{n}$ called a \textit{digit tiling}. It is staightforward to
show that, up to a rigid motion of ${\mathbb{R}}^{n}$,  a digit tiling does
not depend on $\theta$ as long as it is full. Examples of digit tilings, for
example by the twin dragon, appear in numerous books and papers on fractals.
Under fairly mild assumptions \cite[Theorem 4.3]{V1}, a digit tiling is a
tiling by translation by the integer lattice $\mathbb{Z}^{n}$ with the
following self-replicating property: for any tile $t\in T_{L,D}$, it's image
$L(t)$ is the union of tiles in $T_{L,D}$. For this reason, such a tiling is
often referred to as a \textit{reptiling} of ${\mathbb{R}}^{n}$.
\end{example}

\begin{example}
(Crysallographic tilings of ${\mathbb{R}}^{n}$) Gelbrich \cite{Ge} generalized
digit tiling from the lattice group $\mathbb{Z}^{n}$ to any crystallographic
group $\Gamma$. Let $L \, : \; {\mathbb{R}}^{n} \rightarrow{\mathbb{R}}^{n}$
be an expanding linear map such that $L\Gamma L^{-1} \subset\Gamma$. If $D =
\{d_{1}, \dots, d_{N}\}$ is a set of right coset representatives of
$\Gamma/L\Gamma L^{-1}$, then
\[
F = \{{\mathbb{R}}^{n} ; \, L^{-1} d_{1}, \dots, L^{-1} d_{N}\}
\]
is a contractive IFS with attractor $T(\Gamma,L,D)$ with non-empty interior,
called a \textit{crystallographic tile}. The Levy curve is an example of such
a crystallogrpahic tile (for the 2-dimensional crystallographic group $p4$). A
tiling $T_{\theta}$ is called a \textit{crystallographic reptiling}.
\end{example}

\begin{example}
\label{ex:chair}  (Chair tilings of ${\mathbb{R}}^{2}$) The IFS $F =
\{{\mathbb{R}}^{2} ;\, f_{1}, f_{2}, f_{3}, f_{4}\}$ where
\[
\begin{aligned} f_1(x,y) &= (x/2,y/2), \\ f_2(x,y) & = (x/2+1/4, y/2+1/4), \end{aligned}
\qquad\qquad
\begin{aligned} f_3(x,y) & = (-x/2 + 1,y/2), \\ f_4(x,y) & = (x/2,-y/2+1), \end{aligned}
\]
is an IFS whose attractor is a chair or "L"-shaped polygon; see
Figures~\ref{ch1} and \ref{ch2}.  The chair tilings are usually obtained by
what is referred to as a ``substitution method". For the chair tile, there are
uncountably many distinct (non-isometric) tilings $T_{\theta}$. There are
numerous other such polygonal tiles that are the attractors of just touching IFSs.
\end{example}

\begin{figure}[ptb]
\centering
\fbox{\includegraphics[width=12cm, keepaspectratio]
{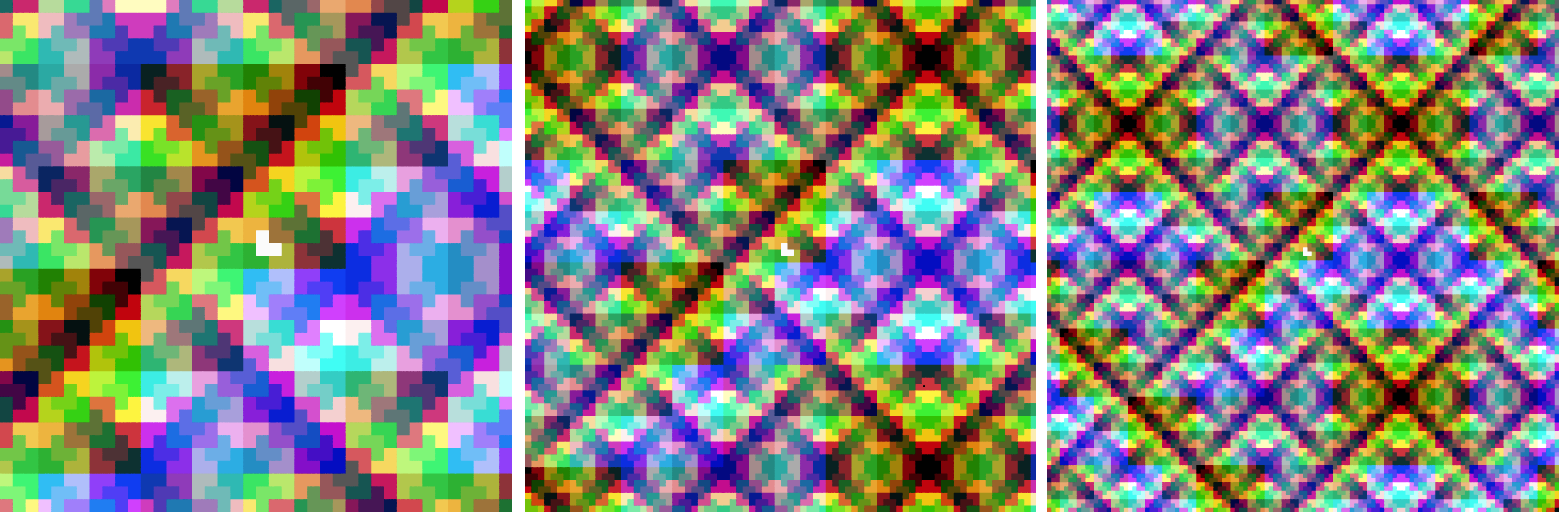}} \caption{Three views of the tiling in
Example~\ref{ex:chair} using $\theta= \overline{12301230}$. The L-shaped
attractor is shown in white near the center. The viewing windows are centered
at the origin, and are of width and height 20,40, and 60 from left to right.}%
\label{ch1}%
\end{figure}

\begin{figure}[ptb]
\centering
\fbox{\includegraphics[
natheight=14.221800in,
natwidth=14.221800in,
height=1.57in,
width=1.57in
]{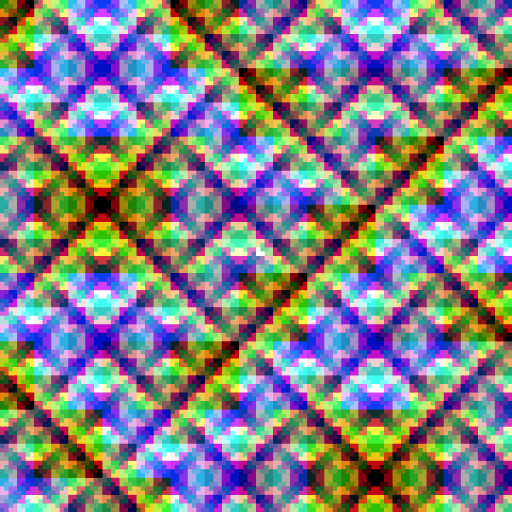}}\caption{Another tiling from the same IFS of
Example~\ref{ex:chair}, but using $\theta= \overline{12300312}$. }%
\label{ch2}%
\end{figure}

\begin{example}
(Fold out tiling) \label{ex:FO1} Let $E$ denote a point in the interior of the
filled square $\blacksquare=[0,1]^{2}$ with vertices $ABCD$. Let $P,Q,R,S$ be
the orthogonal projection of $E$ on $AB,BC,CD,DA$ respectively. Four affine
maps are uniquely defined by $f_{1}(ABCD)=APES,$ $f_{2}(ABCD)=BPEQ,$
$f_{3}(ABCD)=CREQ,$ and $f_{4}(ABCD)=DRES$. The attractor of $F_{E}
=\{\mathbb{R}^{2};f_{1},f_{2},f_{3},f_{4}\}$ is $\blacksquare$. The IFS
$F_{E}$ is of the type introduced in \cite{Ba}, pairs of which may be used to
describe fractal homeomorphisms on $\blacksquare$ (see Example~\ref{ex:FO2} of
Section~\ref{sec:FT}).  According to Theorem~\ref{thm:basinN}, for any full
word $\theta\in\{1,2,3,4\}^{\infty}$, we obtain a tiling $T_{\theta}$ of
${\mathbb{R}}^{2}$, one of which is shown in Figure~\ref{sqtilingthird01},
where $E=(2/3,1/3)$.

Since $a\in f_{i}(\blacksquare)\cap f_{j}(\blacksquare)\neq\emptyset$ implies
that $f_{i}^{-1}(a)=f_{j}^{-1}(a)$, the  mapping $T:$ $\blacksquare
\rightarrow$ $\blacksquare$ given by  $T\left(  x\right)  =f_{i}^{-1}(x)$ when
$x\in f_{i}($ $\blacksquare)$ is well defined and continuous. It can be said
that $f_{i}^{-1}$ applied to $\blacksquare$ causes $\blacksquare$ to be
``continuously folded out from $\blacksquare$". The tilings $T_{\theta}$ can
be thought of as repeated applications of such folding-outs.

\begin{figure}[tbh]
\centering
\fbox{\includegraphics[
natheight=14.221800in,
natwidth=14.221800in,
height=2.1603in,
width=2.1603in
]{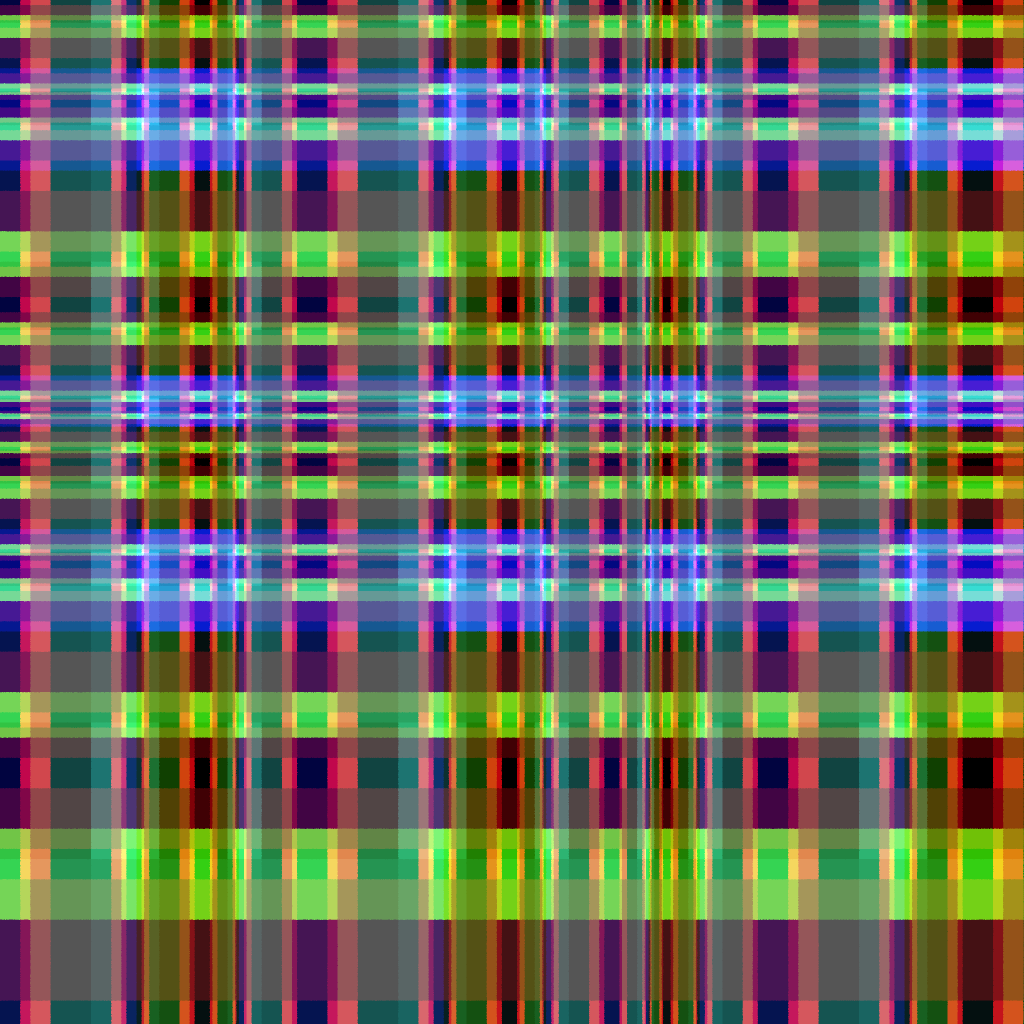}} \caption{The tiling of Example~\ref{ex:FO1}}%
\label{sqtilingthird01}%
\end{figure}
\end{example}

\begin{example}
\label{ex:tri} (Triangular affine and projective tilings)
Figure~\ref{twotilings} shows two views of the same affine tiling of
$\mathbb{R}^{2}$. As in Example~\ref{ex:FO1},  this tiling can be used to
extend a fractal homeomorphism between two triangular attractors to a fractal
homeomorphism of the Euclidean plane (see Section~\ref{sec:FT}). Consider the
IFS $F =\{{\mathbb{R}}^{2};f_{1},f_{2},f_{3},f_{4}\}$ where each $f_{n}$ is an
affine transformations defined as follows. Let $A$, $B$, and $C$ denote three
noncollinear points in $\mathbb{R}^{2}$. Let $c$ denote a point on the line
segment $AB$, $a$ a point on the line segment $BC$, and $b$ a point on the
line segment $CA$, such that $\{a,b,c\}\cap\{A,B,C\}=\emptyset$; see panel (i)
of Figure~\ref{homtrisel}. Let $f_{1}:\mathbb{R}^{2}\rightarrow\mathbb{R}^{2}$
denote the unique affine transformation such that
\[
f_{1}(ABC)=Abc\text{,}%
\]
by which we mean that $f_{1}$ maps $A$ to $A$, $B$ to $b$, and $C$ to $c$.
Using the same notation, let affine transformations $f_{2}$, $f_{3}$, and
$f_{4}$ be the ones uniquely defined by
\[
\begin{aligned}
f_{2}(ABC) &=aBc, \\ f_{3}(ABC) &=abC, \\ f_{4}(ABC) &=abc.
\end{aligned}
\]
Panel (ii) of Figure~\ref{homtrisel} shows the images of the points $A,B,C$
under the four functions of the IFS, illustrating the special way that the
four functions fit together. The attractor of $F$ is the filled triangle with
vertices at $A$, $B$, and $C$, which we will denote by $\blacktriangle$,
illustrated in (iii) in Figure \ref{homtrisel}.

\begin{figure}[tbh]
\includegraphics[width=5in, keepaspectratio]{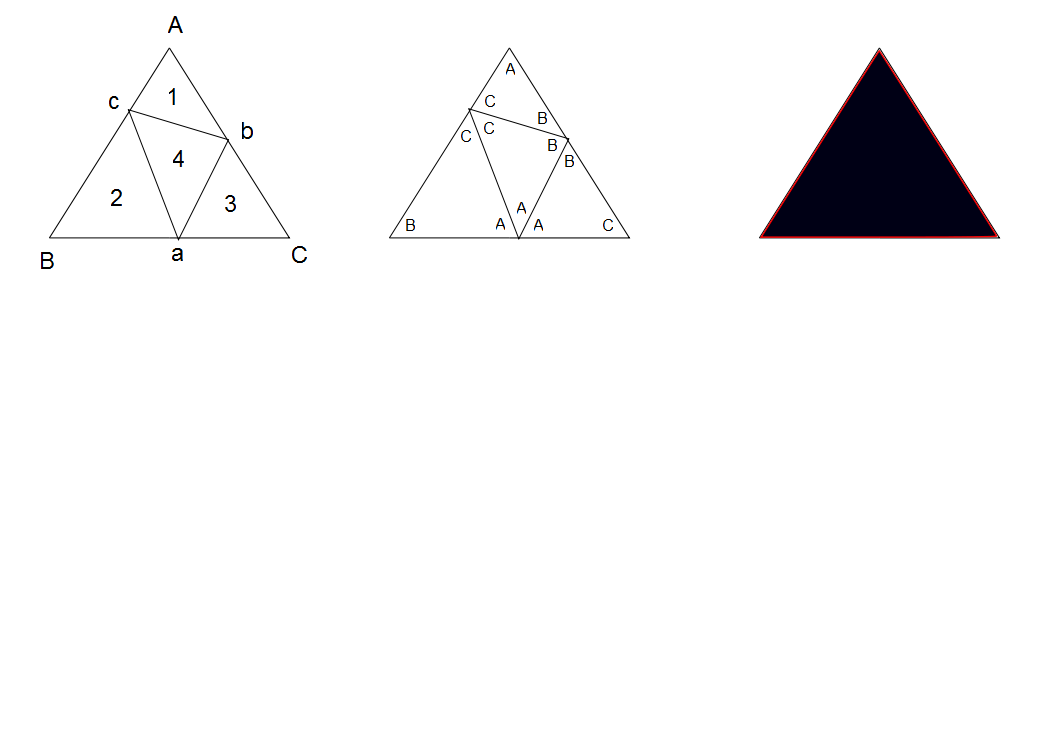} \vskip -5cm
\caption{(i) The points used to define the affine transformations of the IFS
$F=\{\mathbb{R}^{2};f_{1},f_{2},f_{3},f_{4}\}$; (ii) images of the the
triangle $ABC$; (iii) the attractor of the IFS $\{\mathbb{R}^{2};\,
f_{1},f_{2},f_{3}, f_{4}\}$.}%
\label{homtrisel}%
\end{figure}

\begin{figure}[tbh]
\centering
\fbox{\includegraphics[
natheight=14.221800in,
natwidth=28.666800in,
height=2.5175in,
width=5.0462in
]{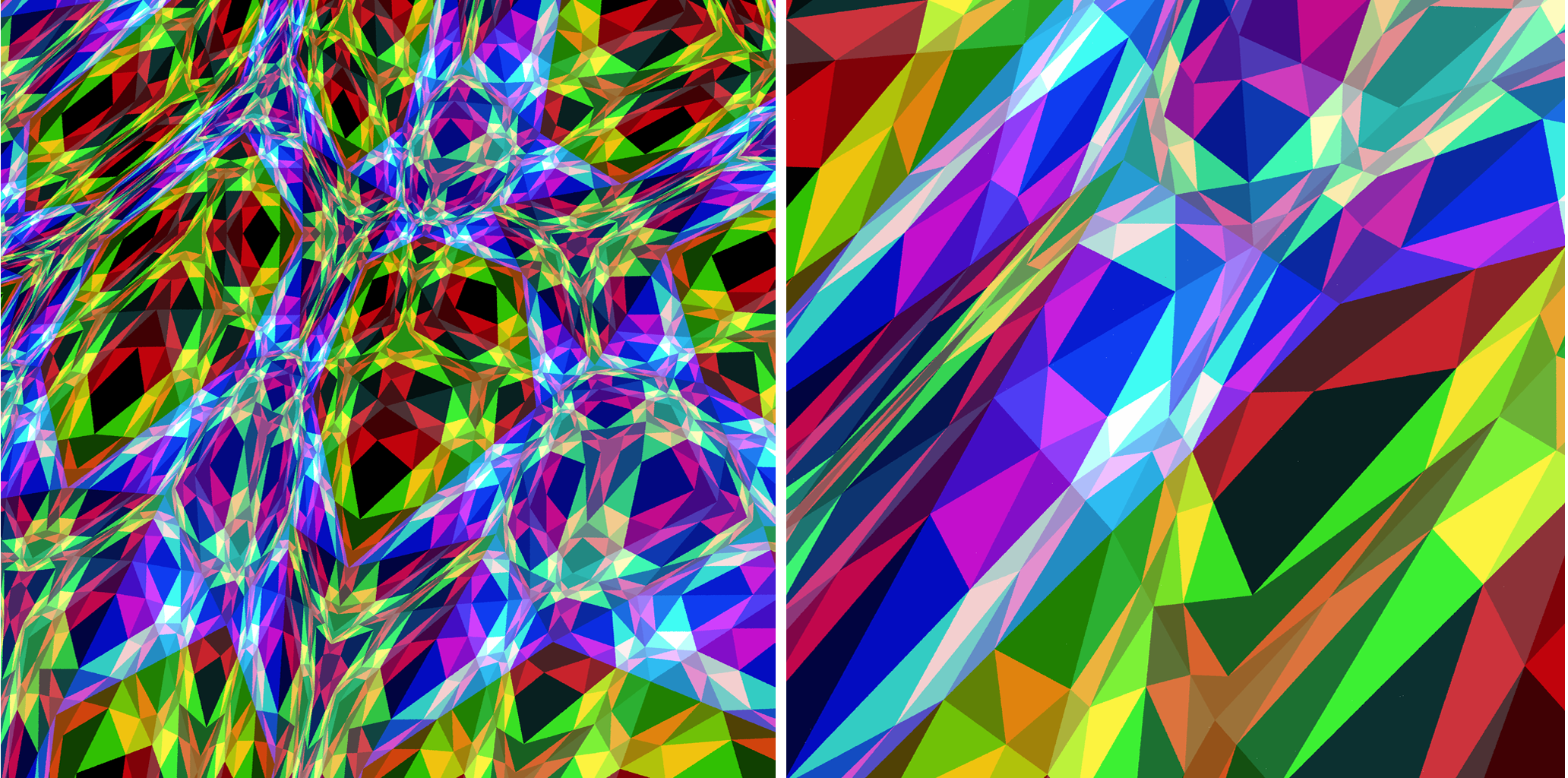}}\caption{The tiling of Example~\ref{ex:tri}. Each image
shows a portion of the same tiling of $\mathbb{R}^{2}$ generated by an affine
IFS with a triangular attractor. All tiles are triangles (the black
quadrilateral is the union of two black triangular tiles).}%
\label{twotilings}%
\end{figure}

If $\theta=\theta_{1}\theta_{2}\theta_{3} \cdots\in\{1,2,3,4\}^{\infty}$ is
full, then according to Theorem~\ref{thm:basinN}, $T_{\theta}$ is a tiling of
${\mathbb{R}}^{2}$ by triangles. One such affine tiling is illustrated in
Figure \ref{twotilings} at two resolutions. A related projective tiling is
shown in Figure~\ref{tilings}. Note that, if $e$ is a common edge of two
triangles $\Delta_{1},\Delta_{2}$ in $T_{\theta}$, then $e$ is the image of
the same edge of the original triangle $ABC$ from $\Delta_{1}$ and $\Delta
_{2}$, and if $v$ is a common vertex in $T_{\theta}$ of two triangles
$\Delta_{2},\Delta_{2}$ in $T_{\theta}$, then $v$ is the image of the same
vertex of the original triangle $ABC$ from $\Delta_{1}$ and $\Delta_{2}$.

\begin{figure}[tbh]
\centering
\fbox{\includegraphics[
width=2.5in, keepaspectratio
]{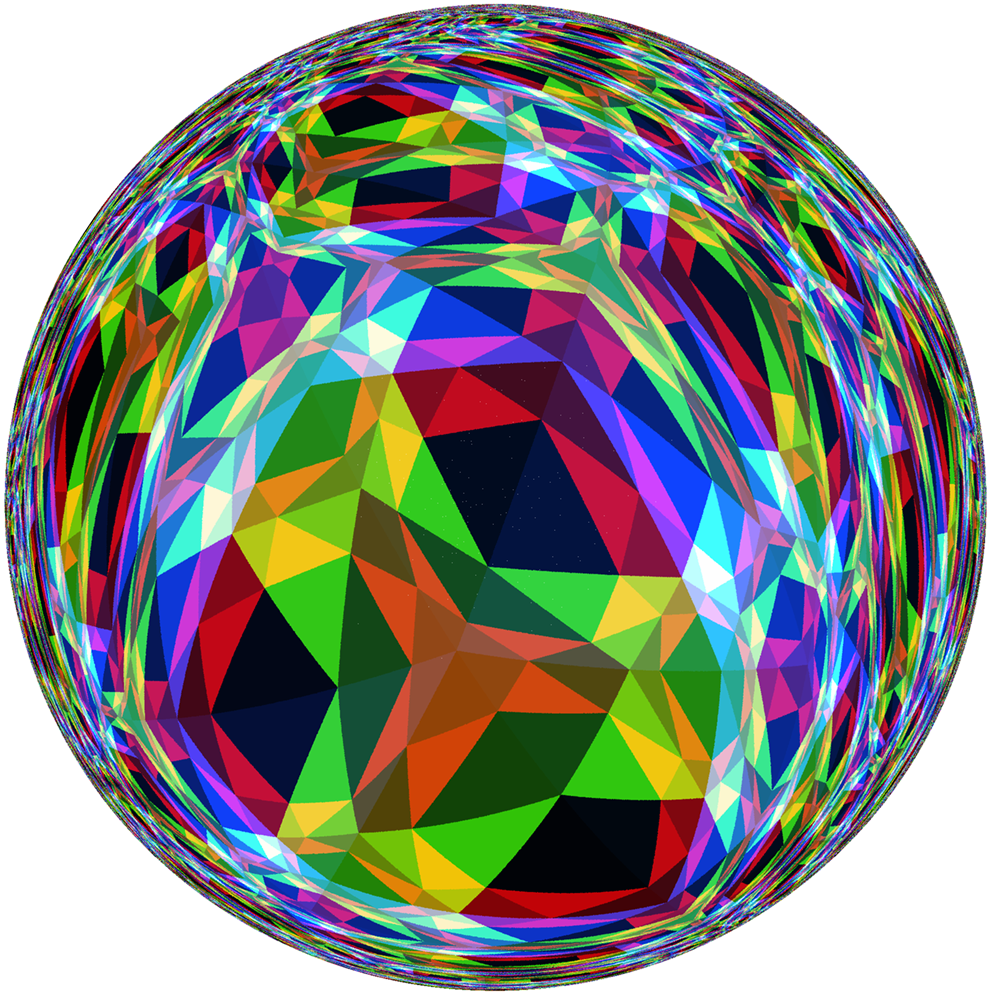}} \vskip 3mm \caption{The image on the left illustrates a
projective tiling constructed an a fashion analogous to that of the tiling in
Figure~\ref{twotilings}. It is represented using the disk model for
$\mathbb{R}P^{2}$.}%
\label{tilings}%
\end{figure}
\end{example}

Although Theorem~\ref{thm:basinN} and Theorem~\ref{thm:random} hold for IFSs
whose atttractor has nonempty interior, the basic tiling construction of this
section applies as well to IFSs with an attractor with empty interior.

\begin{example}
\label{ex:empty1}(Tilings from an attractor with empty interior)
Figure~\ref{fig:Sier} shows a tiling by copies of the Sierpinski gasket.
\begin{figure}[tbh]
\centering
\fbox{\includegraphics[width=2.5in, keepaspectratio]{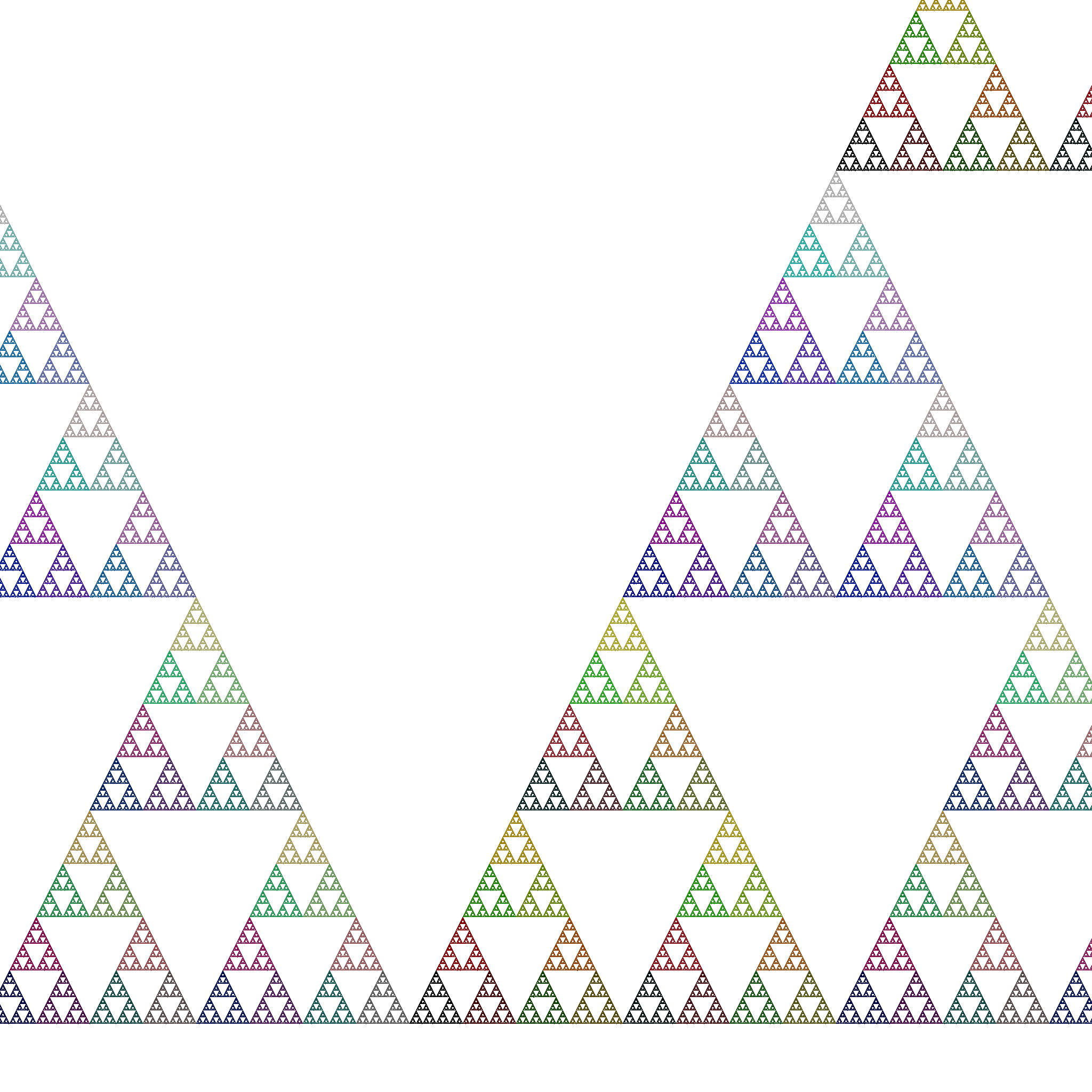}}
\vskip 3mm\caption{The original attractor is the small black Sierpinski gasket
near the middle bottom; $\theta=322222\cdots$.}%
\label{fig:Sier}%
\end{figure}

Figures \ref{nett01} and \ref{nettz01} show the tiling $T_{\theta}$, for a
particular $\theta$, for the IFS $\left\{  \mathbb{R}^{2};f_{i}%
,i=1,2,3\right\}  $ where%
\begin{align*}
f_{1}(x,y)  &  =(-0.7x+0.7,0.65y+0.35),\\
f_{2}(x,y)  &  =(-0.3y+1,-0.6x-0.3y+1.3),\\
f_{3}(x,y)  &  =\left(  0.375y+0.325,-0.6x+0.35y+0.65\right)  .
\end{align*}

\end{example}

\begin{figure}[ptb]
\centering
\fbox{\includegraphics[
natheight=3.058800in,
natwidth=3.058800in,
height=3.0588in,
width=3.0588in]{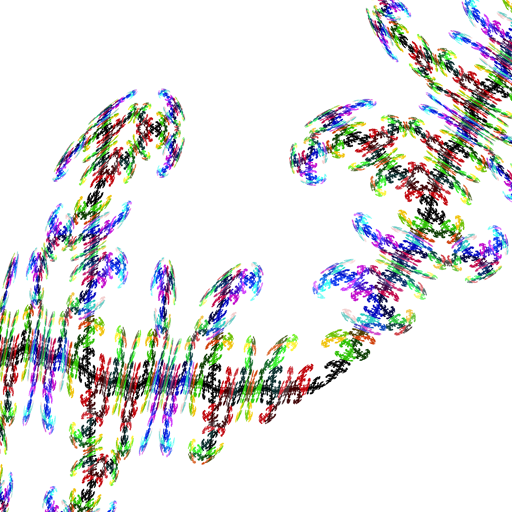}}\caption{The tiling of Example~\ref{ex:empty1}}%
\label{nett01}%
\end{figure}

\begin{figure}[ptb]
\vskip 1cm \centering
\fbox{\includegraphics[
natheight=3.394400in,
natwidth=3.083100in,
height=3.3944in,
width=3.0831in]{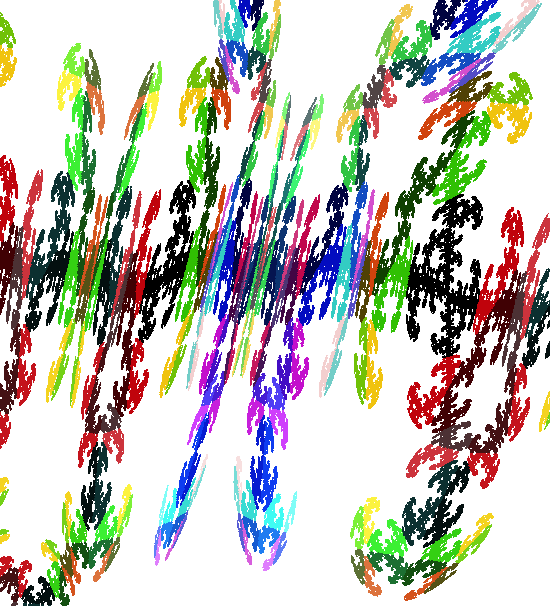}}\caption{Zoom of the tiling in
Figure~\ref{nett01}}%
\label{nettz01}%
\end{figure}

\begin{example}
(A tiling of a fractal continuation of a fractal function) A fractal function,
defined say on the unit interval, is a function, often everywhere
non-differentiable, whose graph is the attractor of an IFS.
Figure~\ref{fig:FC} shows a tiling, obtained by the method of this section, by
copies of the graph of such a function. Such extensions of fractal functions
are essential to the notion of a fractal continuation, a generalization of
analytic continuation of an analytic function; see \cite{BV3}.

\begin{figure}[tbh]
\vskip -3mm
\centering
\fbox{\includegraphics[width=2.5in, keepaspectratio]{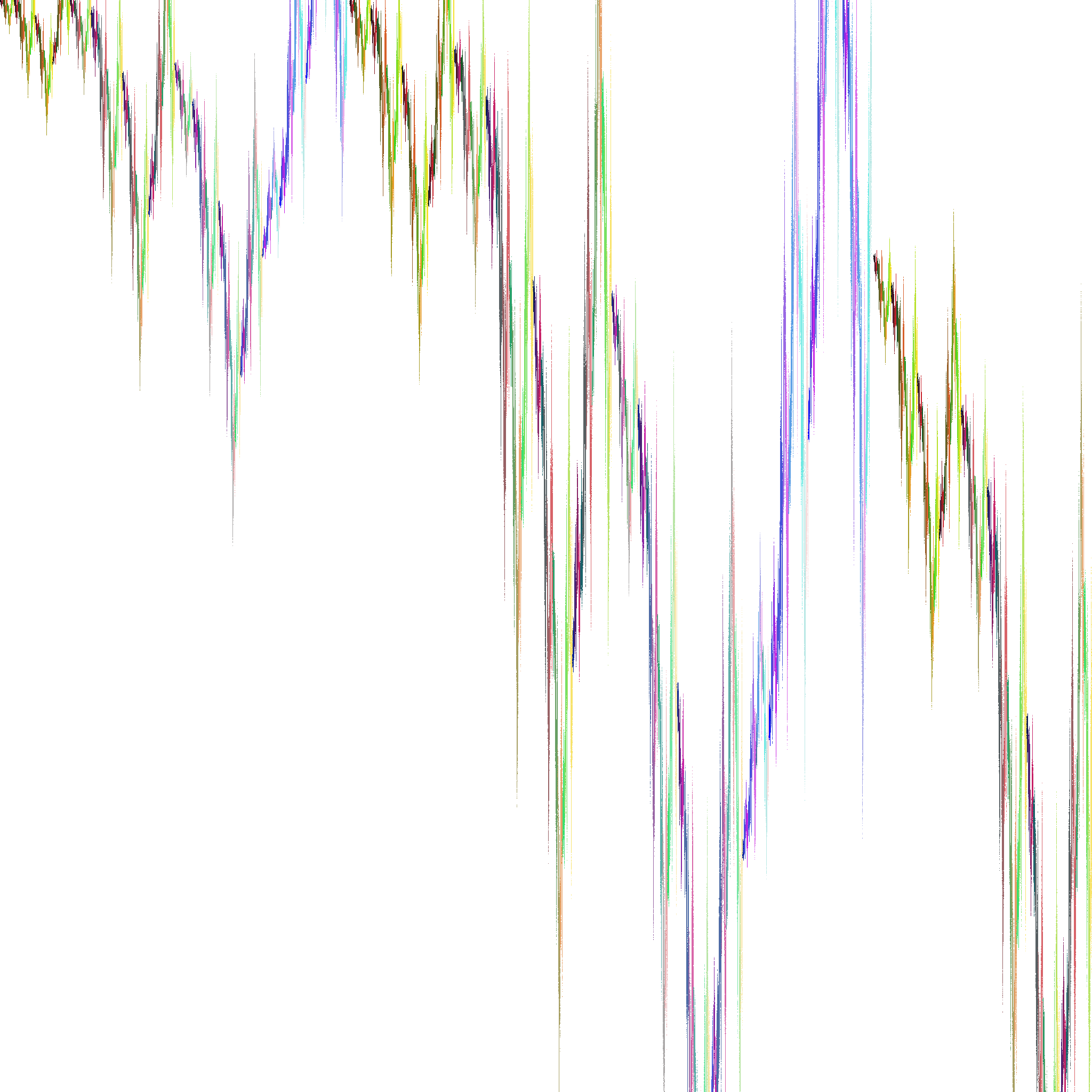}}
\vskip 3mm\caption{A tiling of the graph of a continuation  \cite{BV3} of a
fractal function by copies of the graph of the fractal function.}%
\label{fig:FC}%
\end{figure}
\end{example}
\vskip -5mm

\section{Fractal Tilings from an Overlapping IFS}

\label{sec:O}

In this section the concept of a mask is used to generalize the tilings of the
previous section from the non-overlapping to the overlapping case.

\begin{definition}
\label{def:mask} For an IFS $F =\left(  \mathbb{X};f_{1},f_{2},...,f_{N}%
\right)  $ with attractor $A$, a \textbf{mask} ${\mathcal{M}} = \{M_{i}, 1\leq
i \leq N\}$ is a tiling of $A$ such that $M_{i} \subseteq f_{i}(A)$ for all
$f_{i}\in F$.
\end{definition}

Let $F$ be an invertible IFS with attractor $A$, and let $\theta\in\lbrack
N]^{\infty}$. Let ${\mathcal{M}}=\{M_{1},M_{2},\dots,M_{N}\}$ be any mask of
$A$ with the property that $M_{\theta_{1}}=f_{\theta_{1}}(A)$. Define
recursively a sequence of masked IFSs $F_{n}=\{{\mathbb{X}};\,f_{n,1}%
,f_{n,2},\dots,f_{n,N}\}$ with respective masks ${\mathcal{M}}_{n}$ and
attractors $A_{n}$ and associated tilings $T_{n}$, for $n=1,2,\dots$, as
follows: $F_{1}=F,\,A_{1}=A,\,T_{1}=\{A\},\,{\mathcal{M}}_{1}={\mathcal{M}}$,
and
\begin{equation}
\begin{aligned} T_{n+1} &:= \left\{ (f_{n,\theta_n}^{-1} ( f_{n,i} (t) \cap M_{n,i}) \,:\, i=1,2,\dots, N,\, t\in T_n\right \} \\ F_{n+1} &:= \left \{ {\mathbb X}; \ f_{n,\theta_n}^{-1}\circ f_{n,i}\circ f_{n,\theta_n} \, : \, i=1,\dots ,N \right \}, \\ A_{n+1} &:= f^{-1}_{n,\theta_n}(A_n) \\ {\mathcal M}_{n+1} &:= \{ M_{n+1,1}, \dots , M_{n+1,N}\}, \quad \text{where} \\ & \begin{cases} M_{n+1,\theta_{n+1}} = f_{n+1,\theta_{n+1}}(A_{n+1}), \\ M_{n+1,j} = f_{n,\theta_n}^{-1}(M_{n,j}) \setminus f_{n+1,\theta_{n+1}}(A_{n+1}) \quad \text{if }\; j \neq \theta_{n+1} \end{cases} \\ \ddots \end{aligned}\label{eq:over}%
\end{equation}
The following proposition is not hard to verify.


\begin{prop}
For all $n=1,2,\dots$, we have

\begin{enumerate}
\item $A_{n}$ is the attractor of $F_{n}$;

\item ${\mathcal{M}}_{n}$ is a mask for $A_{n}$;

\item $T_{n}\subset T_{n+1}$.
\end{enumerate}
\end{prop}

In light of statement (3), define, with respect to $F$ and $A$:
\[
T_{\theta} := T_{{\mathcal{M}},\theta} = \bigcup_{n=1}^{\infty}T_{n}.
\]
If an attractor $A$ of an IFS $F$ is non-overlapping, then
\[
{\mathcal{M}} := \left\{  f_{i}(A) \, : \, i = 1,2, \dots, N \right\}
\]
is a mask for $A$. It is not hard to verify that, in this case, the tiling
$T_{{\mathcal{M}},\theta}$ is exactly the tiling $T_{\theta}$ defined by
Equation~\eqref{eq:T} in Section~\ref{sec:non-overlapping}.

\begin{theorem}
\label{thm:basinO} Let $F=\left(  \mathbb{X}\, ; \, f_{1},f_{2},...,f_{N}%
\right) $ be an invertible IFS with attractor $A$ with non-empty interior and
mask $\mathcal{M}$. If $\theta$ is a full word, then the tiling
$T_{{\mathcal{M}, \theta}}$ covers the basin $B(A)$. If $F$ is contractrive,
then $T_{{\mathcal{M}, \theta}}$ is a full tiling of ${\mathbb{X}}$.
\end{theorem}

\begin{proof}
It is sufficient to show that the basin $B:=B(A)$ is covered by the tiling,
so let $x\in B$. According to the third line of Equation~\eqref{eq:over}, the
tiling $T_{{\mathcal{M}, \theta}}$ covers
\[
\begin{aligned} & \cdots \left [ ( f_{\theta_1}^{-1}  f_{\theta_2}^{-1}  f_{\theta_1} )(  f_{\theta_1}^{-1}  f_{\theta_3}^{-1}  f_{\theta_1} ) (  f_{\theta_1}^{-1}  f_{\theta_2}  f_{\theta_1})\right ](  f_{\theta_1}^{-1}  f_{\theta_2}^{-1}  f_{\theta_1}) f_{\theta_1}^{-1} (A)  \\ & \hskip 10mm =  f_{\theta_1}^{-1}  f_{\theta_2}^{-1}  f_{\theta_3}^{-1} \cdots  f_{\theta_n}^{-1} (A) \end{aligned}
\]
for any $n$. So it sufices to show that $x \in f_{\theta_{1}}^{-1}
f_{\theta_{2}}^{-1} f_{\theta_{3}}^{-1} \cdots f_{\theta_{n}}^{-1} (A)$ for
some $n$, or equivalently
\[
f_{\theta_{n}} f_{\theta_{n-1}} \cdots f_{\theta_{1}} (x) \in A.
\]

By the definition of attractor, for any $\epsilon> 0$ there is an
$M_{\epsilon}$ such that if $m \geq M_{\epsilon}$, then
\[
F^{m}(x) \subset A_{\epsilon},
\]
where $A_{\epsilon}$ is the open $\epsilon$-neighborhood of $A$. Because
$\theta$ is assumed to be full, there exists a compact $A^{\prime}$ with
$A^{\prime o}$ with the property that, for any $M$ there exist $n>m\geq M$
such that
\[
f_{\theta_{n}}\circ\cdots\circ f_{\theta_{m+1}}(A) \subset A^{\prime o}.
\]
This implies that there exists an $\epsilon_{0}$-neighborhood $A_{\epsilon
_{0}}$ of $A$ such that
\[
f_{\theta_{n}}\circ\cdots\circ f_{\theta_{m+1}}(A_{\epsilon_{0}}) \subset
A^{o},
\]
for some $\epsilon_{0}>0$. If $M \geq M_{\epsilon_{0}}$, then
\[
f_{\theta_{n}}\circ\cdots\circ f_{\theta_{m+1}}\circ f_{\theta_{m}}
\circ\cdots\circ f_{\theta_{1}}(x) \in f_{\theta_{n}}\circ\cdots\circ
f_{\theta_{m+1}}(A_{\epsilon_{0}}) \subset A,
\]
as required.
\end{proof}

\begin{example}
A portion of a one-dimensional masked tiling, illustrated in Figure
\ref{tilingstrip002}, is generated by the overlapping IFS
\[
F=\left\{  \mathbb{R}:f_{1}(x)=bx,f_{2}(x)=bx+(1-b)\right\}
\]
with $b=0.65$. The unique attractor of $F$ is $A=[0,1]$. The mask for $A$
(w.r.t. $F$) is $\{M_{1}=[0,b],M_{2}=[b,1]\}$. The left-most tile is black and
corresponds to the interval $A$. \medskip

In this and other pictures, different colors represent different tiles. Some
colors may be close together. The images are approximate.
\end{example}

\begin{figure}[tbh]
\centering
\fbox{\includegraphics[
natheight=3.555200in,
natwidth=7.111400in,
height=0.528in,
width=2.0357in
]{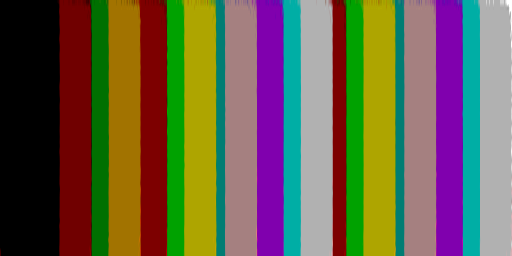}} \caption{Part of a tiling of $\mathbb{R}$ generated by
a masked IFS.}%
\label{tilingstrip002}%
\end{figure}

\begin{example}
Figures \ref{tilingsquare002} ($b=0.65$) and \ref{sqtilingpt9pt} ($b=0.9$)
illustrate masked tilings of $\mathbb{[}0,\infty)^{2}\subset\mathbb{R}^{2}$
associated with the family of IFSs $F=\left\{  \mathbb{[}0,\infty)^{2}%
:f_{1},f_{2},f_{3},f_{4}\right\}  $ where%
\[
\begin{aligned}
f_{1}(x,y)  &  = (bx,by), \\  f_{3}(x,y)  &  = (bx,by+l), \end{aligned} \qquad
\begin{aligned} f_{2}(x,y) &=(bx+l,by) \\ f_{4}(x,y) &= (bx+l,by+l)
\end{aligned}
\]
with $l=(1-b)$. The attractor is the filled unit square $A=\blacksquare
=[0,1]^{2}$ represented by the black tile in the upper left corner of both
images. The mask is
\[
\begin{aligned}
M_{1}&=f_{1}(\blacksquare), \\
M_{k+1}&=f_{k+1}(\blacksquare)\backslash\cup
_{i=1}^{k}M_{i}\text{ }\quad (k=1,2,3). \end{aligned},
\]
which is referred to as the \textit{tops mask}. In both cases the tiling is
generated by the string $\theta=\overline{1}$.

\begin{figure}[ptbh]
\centering
\fbox{\includegraphics[
trim=0.000000in 0.000000in -0.003555in 0.000000in,
natheight=7.111400in,
natwidth=7.111400in,
height=2.0349in,
width=2.0357in
]{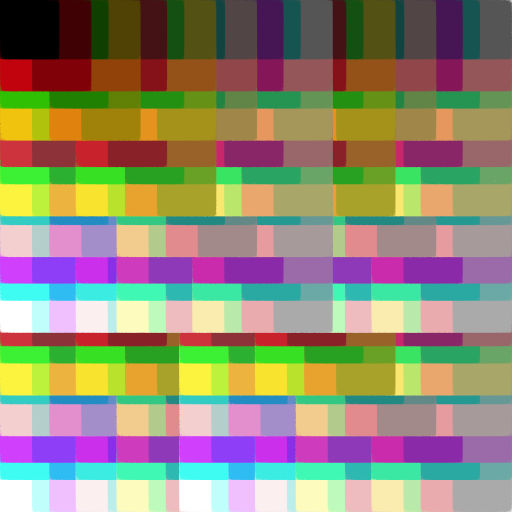}} \caption{Part of a two-dimensional masked tiling. The
top row shows the same one-dimensional masked tiling illustrated in Figure
\ref{tilingstrip002}.}%
\label{tilingsquare002}%
\end{figure}

\begin{figure}[ptb]
\centering
\fbox{\includegraphics[
natheight=14.221800in,
natwidth=14.221800in,
height=2.0182in,
width=2.0182in
]{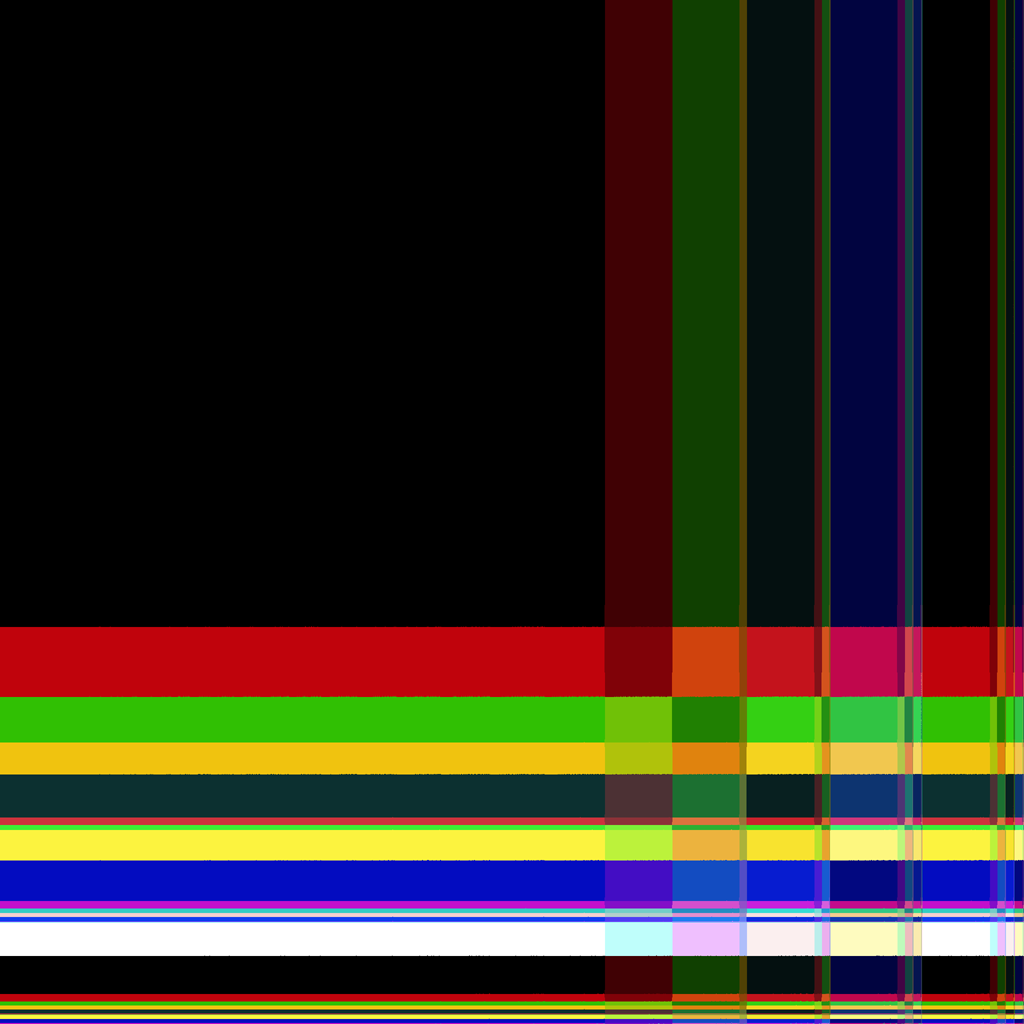}} \caption{A masked tiling of a quadrant of the Euclidean
plane, similar to Figure \ref{tilingsquare002}, but with scaling factor
$b=0.9$.}%
\label{sqtilingpt9pt}%
\end{figure}
\end{example}

\section{Tilings from a Graph IFS}

\label{sec:GIFS}

The construction of tilings from an IFS can be generalized to the construction
of tilings from a graph IFS. There is considerable current interest in graph
IFS and tilings corresponding to Rauzy fractals \cite{jolivet, rao, rauzy}.
Hausdorff dimension of attractors associated with graph IFSs has been
considered in \cite{M}. Graph IFSs arise in connection with substitution
tilings and number theory; see for example \cite{akiyama, berthe}.

Let ${\mathbb{H}}^{M}$ denote the $N$-fold cartesian product of $M$ copies of
${\mathbb{H}}(\mathbb{X })$. A \textit{graph iterated function system} (GIFS)
is a directed graph $G$, possibly with loops and multiple edges in which the
vertices of $G$ are labeled by $\{ 1,2, \dots, M\}$ and each directed edge $e$
is labeled with a continuous function $f_{e}: \mathbb{X }\rightarrow
\mathbb{X}$. It is also assumed that $G$ is \textit{strongly connected}, i.e.,
that there is a directed path from any vertex to any other. Let $E_{ij}$
denote the set of edges from vertex $i$ to vertex $j$. Define the function
\[
F \; : \;{\mathbb{H}}^{M} \rightarrow{\mathbb{H}}^{M}%
\]
as follows. If $\mathbf{X} = (X_{1},X_{2}, \dots, X_{M}) \in{\mathbb{H}}^{M}$,
then
\[
F(\mathbf{X}) = (F_{1}(\mathbf{X}), F_{2}(\mathbf{X}), \dots, F_{M}%
(\mathbf{X})),
\]
where
\[
F_{i}(\mathbf{X}) = \bigcup_{j=1}^{M} \;\bigcup_{e\in E_{ij}} f_{e}(X_{j})
\]
for $i=1,2,\dots M$. It can be shown that, if each $f_{e}$ is a contraction,
then $F$ is a contraction on ${\mathbb{H}} ^{M}$, and consequently has a
unique fixed point or attractor $\mathbf{A} = (A_{1}, A_{2}, \dots, A_{M})$.
The concecpts of non-overlapping and invertible are defined exactly as for an
IFS. In fact, an ordinary IFS is the special case of a graph IFS where $G$ has
exactly one vertex and all the edges are loops.

Assume that the graph IFS is a non-overlapping and invertible with attractor
$\mathbf{A} = (A_{1}, A_{2}, \dots, A_{M})$. Let $G^{\prime}$ denote the graph
obtained from $G$ by reversing the directions on all of the edges. For any
directed (infinite) path $\theta=e_{1}e_{2}\cdots$ in $G^{\prime}$, a tiling
is constructed as follows. First extend previous notation so that $\theta
_{k}=e_{k},\;\theta|k=e_{1}e_{2}\cdots e_{k}$ and $f_{\theta|k}=f_{e_{1}}\circ
f_{e_{2}}\circ\cdots\circ f_{e_{k}},\; (f^{-1})_{\theta|k} :=f_{e_{1}}%
^{-1}\circ f_{e_{2}}^{-1}\circ\cdots\circ f_{e_{k}}^{-1}$. Given any directed
path $\omega$ of length $k$ in $G$ that starts at the vertex at which
$\theta|k$ terminates, let
\[
\begin{aligned} t_{\theta,\omega} &= ((f^{-1})_{\theta|k}\circ f_{\omega})(A_j), \\
T_{\theta,k} &= \{ t_{\theta,\omega} \, : \, \omega \in W_{k} \}, \end{aligned}
\]
where $j$ is the terminal vertex of the path $\omega$, and $W_{k}$ is the set
of directed paths of length $k$ in $G$ that start at the vertex at which
$\theta|k$ terminates. Since, for any $\omega\in W_{k}$, we have
\[
(f^{-1})_{\theta|k}\circ f_{\omega}=(f^{-1})_{\theta|k}\circ(f_{\theta_{k+1}%
})^{-1}\circ f_{\theta_{k+1}}\circ f_{\omega}=(f^{-1})_{\theta|k+1}\circ
f_{\theta_{k+1}\omega},
\]
the inclusion
\[
T_{\theta,k}\subset T_{\theta,k+1}%
\]
holds for all $k$. Therefore
\[
T_{\theta} :=\bigcup_{k=1}^{\infty}T_{\theta,k}
\]
is a tiling.

\begin{figure}[tbh]
 \includegraphics[width=5in, keepaspectratio]{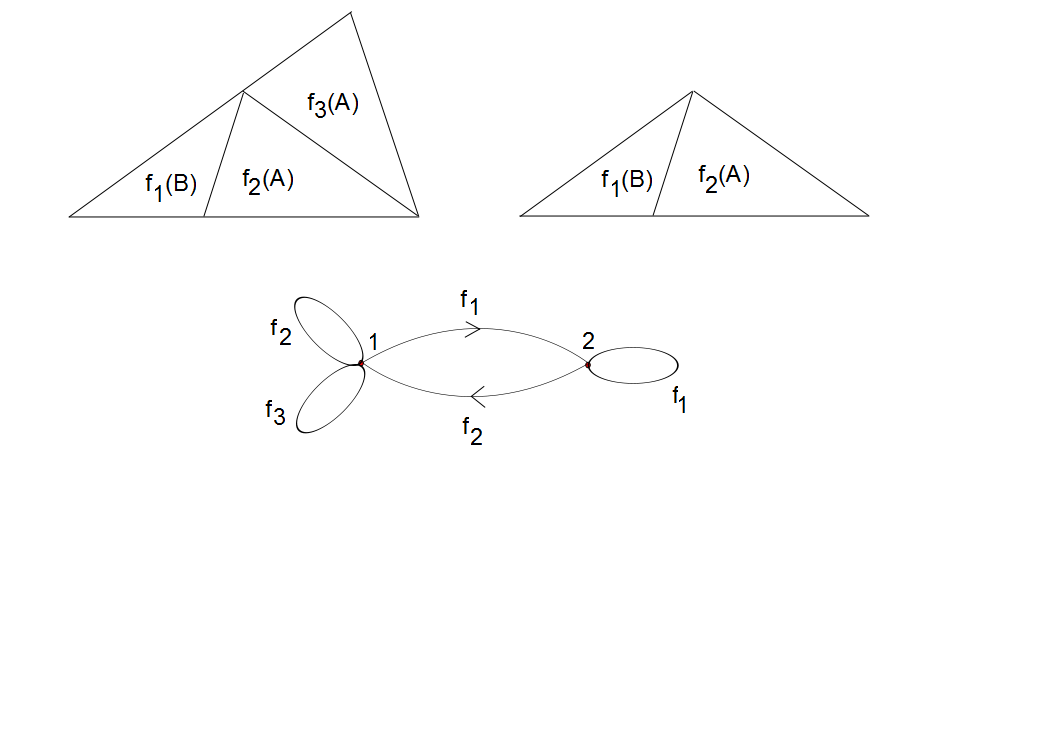}
\vskip -3.5cm\caption{Graph IFS for the Penrose tiles.}
\label{Penrose}
\end{figure}

\begin{example}
(Penrose tilings of ${\mathbb{R}}^{2}$) In this example the graph $G$ is given
in Figure~\ref{Penrose} where, using the complex representation of
${\mathbb{R}}^{2}$, the functions are
\[
f_{1}(z)=\left(  \frac{z}{\tau}+1\right)  \,\omega_{1},\qquad f_{2}%
(z)=\frac{-z}{\tau}+\tau^{2},\qquad f_{3}(z)=\frac{z}{\tau}\,\omega_{3}%
+\tau^{2},
\]
where $\tau=(1+\sqrt{5})/2$ is the golden ratio and $\omega_{k}=\cos
(k\pi/5)+i\,\sin(k\pi/5),\,0\leq k\leq9$, are the tenth roots of unity. The
acute isoceles triangle $A$ and obtuse isoceles triangle $B$ in the figure
have long and short sides in the ratio $\tau:1$, and the angles are
$\pi/5,2\pi/5,2\pi/5$ and $\pi/5,\pi/5,3\pi/5$, respectively. The attractor of
the graph IFS is the pair $(A,B)$:
\[
\begin{aligned} A &= f_1(B)\cup f_2(A) \cup f_3(A), \\
B &= f_1(B) \cup f_2(A). \end{aligned}
\]
Clearly, there are periodic tilings of the plane using copies of these tiles.
However, if $\theta$ is a directed path in $G$ and $T_{\theta}$ tiles
${\mathbb{R}}^{2}$, then this is a non-periodic tiling. Although a Penrose
tiling is usually given in terms of kites and darts or thin and fat rhombs,
these are equivalent to tiling by the acute and obtuse triangles described in
this example.
\end{example}

\begin{example}
(Another GIFS tiling) \label{ex:GIFS} In Figure~\ref{fig:GIFS} the graph has
two vertices coresponding to the two components of the attractor, the
isosceles right triangle $T$ and square $S$ shown at the top. The eight edges
of the the graph, corresponding to the eight functions are shown graphically
by their images, the eight colored triangles and rectangles shown at the top.
Four of the functions map the triangle $T$ onto the four smaller triangles,
and the other four functions map the square $S$ to the two smaller squares and
two rectangles. One of the infinitely many possible tilings that can be
constructed, using the method described of this section, is shown in the
bottom panel.

\begin{figure}[ptb]
\centering
 \hskip 15mm \includegraphics[width=2.5cm, keepaspectratio]{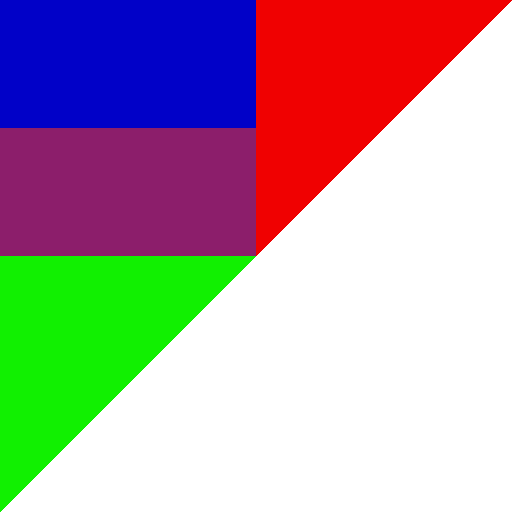}  \hskip 10mm
\includegraphics[width=2.5cm, keepaspectratio] {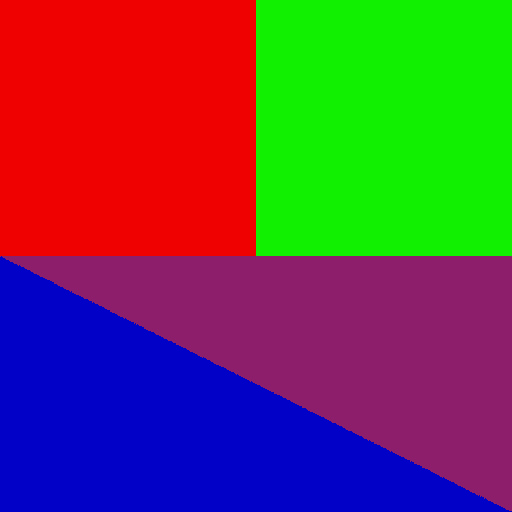}
\newline\vskip 5mm
\fbox{\includegraphics[width=7cm, keepaspectratio] {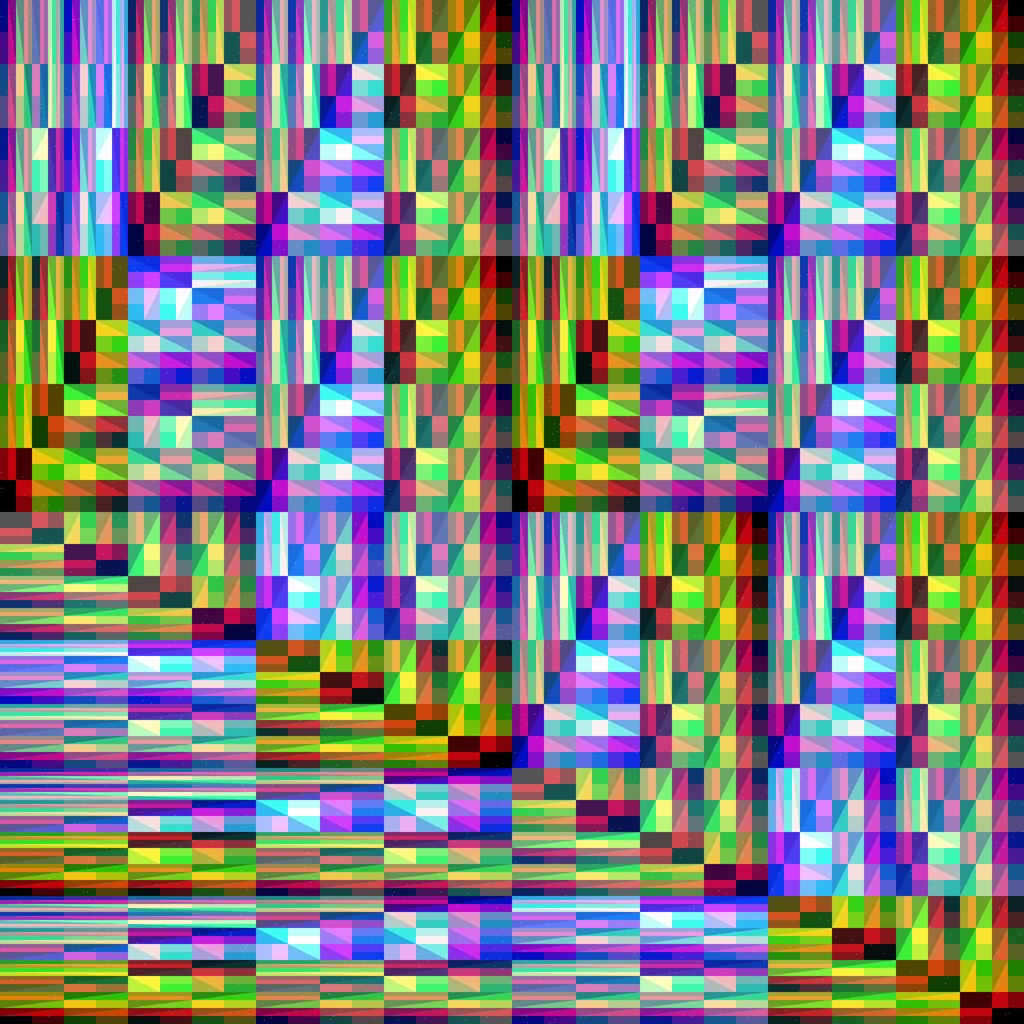}}
\caption{Another GIFS tiling; see Example~\ref{ex:GIFS}.}%
\label{fig:GIFS}%
\end{figure}
\end{example}

\section{Fractal Transformations from Tilings}

\label{sec:FT}

The goal of this section is, given a fractal transformation, to extend its
domain from the attractor of the IFS to the basin of the IFS, for example in
the contractive IFS case, from the attractor to all of ${\mathbb{X}}$.

\subsection{Fractal Transformation}

In Section~\ref{sec:IFS}, each point of a point-fibred attractor an IFS was
assigned a set of addresses.  To choose a particular address for each point of
the attractor, the following notion is introduced.

\begin{definition}
\label{def:section} Let $A$ be a point-fibred attractor of an invertible IFS
$F$. A \textbf{section} of the coordinate map $\pi\, : \, [N]^{\infty}
\rightarrow A$ is a map $\tau\,:\,A\rightarrow[N]^{\infty}$ such that
$\pi\circ\tau$ is the identity. For $x\in A$, the word $\tau(x)$ is referred
to as the \textbf{address} of $x$ with respect to the section $\tau$.
\end{definition}

If $A$ is a point-firbred attractor of $F$, then the following diagram
commutes for all $n\in\lbrack N]$:
\[%
\begin{array}
[c]{ccc}%
[N]^{\infty} & \overset{S_{n}}{\rightarrow} & [N]^{\infty}\\
\pi\downarrow\text{\ \ \ \ } &  & \text{ \ \ \ }\downarrow\pi\\
A & \underset{f_{n}}{\rightarrow} & A
\end{array}
\]
where the inverse shift map $S_{n}\,:\, [N]^{\infty} \rightarrow[N]^{\infty}$
is defined by $S_{n}(\omega)=n\omega$ for $n\in\lbrack N]$.

The notion of fractal transformation was introduced in \cite{tops}; see also
\cite{Ba,BV}. A fractal transformation is a mapping from the attractor of one
IFS to the attractor of another IFS of the following type.

\begin{definition}
Given two IFSs $F$ and $G$ with an equal number of functions, with
point-fibred attractors $A_{F}$ and $A_{G}$, with coordinate maps $\pi_{F}$
and $\pi_{G}$ and with shift invariant sections $\tau_{F}$ and $\tau_{G}$,
each of the maps $\pi_{F} \circ\tau_{G} \, : A_{G} \rightarrow A_{F}$ and
$\pi_{G} \circ\tau_{F} \, : A_{F} \rightarrow A_{G}$ is called a
\textbf{fractal transformation}. If a fractal transformation $h$ is a
homeomorphism, then $h$ is called a \textbf{fractal homeomorphism}.
\end{definition}

A homeomorphism $h \,: A_{F} \rightarrow A_{G}$ is a fractal homeomorphism
with respect to shift invariant sections $\tau_{F}$ and $\tau_{G}$ if and only
if $h$ satisfies the commuting diagram
\begin{equation}
\label{homeo}%
\begin{array}
[c]{ccc}%
A_{F} & \underset{h}{\rightarrow} & A_{G}\\
\text{$\tau_{F}$} \searrow &  & \swarrow\text{$\tau_{G}$}\\
& [N]^{\infty} &
\end{array}
\end{equation}
i.e., the function $h$ takes each point $x\in A_{F}$ with address $\omega=
\tau_{F}(x)$ to the point $y \in A_{G}$ with the same address $\omega=\tau
_{G}(y) $. See \cite{BV}. \medskip

\subsection{Extending the Domain of a Fractal Transformation}

To extend the domain of a fractal transformation, we introduce a ``decimal"
notation $\theta\bullet\omega$, where $\theta$ is a finite word and $\omega$
is an infinite word in the alphabet $[N]$. Let%

\[
\Omega:= \{\theta_{1} \theta_{2} \cdots\theta_{k} \bullet\omega_{1} \omega_{2}
\cdots\, | \, k\in\{0\}\cup\mathbb{N}, \, \theta_{i}, \omega_{i} \in[N] \,
\text{for all} \, i, \, \theta_{k} \neq\omega_{1}\}.
\]
For an IFS $F$ the coordinate map $\pi\, : \, [N]^{\infty} \rightarrow A$ can
be extended as follows.

\begin{definition}
\label{def:excode}  For an invertible IFS on a complete metric space
${\mathbb{X}}$ with point-fibred attractor $A$ and $\theta\bullet\omega
\in\Omega$, define the \textbf{extended coordinate map} $\widehat{\pi} \, : \,
\Omega\rightarrow{\mathbb{X}}$ of $\pi\, : \, [N]^{\infty} \rightarrow A$ by
\[
\widehat{\pi} (\theta\bullet\omega) := (f^{-1})_{\theta}(\pi(\omega)).
\] \vskip 2mm
\noindent Recall the notation $F^{*}:=\left(  \mathbb{X};f_{1}^{-1},f_{2}^{-1},...,f_{N}^{-1}\right)$,
so that  $F^*(X) = \bigcup_{f\in F} f^{-1}(X)$.
The range of $\widehat{\pi}$, called the \textbf{fast basin} of the attractor
$A$ and denoted ${\widehat{B}} = {\widehat{B}}(A)$, is equal to
\[
{\widehat{B}}(A) = \bigcup_{k=0}^{\infty} (F^*)^{k} (A) =\left\{  x
\in{\mathbb{X}} \, : \, f_{\theta}(x) \in A \text{ \; for some\;} \theta
\in\bigcup_{k=1}^{\infty} [N]^{k}\right\} .
\]

\end{definition}

\begin{prop}
\label{prop:fast} The fast basin $\widehat{B}$ of a point-fibred attractor $A$
of an invertible IFS $F$ on ${\mathbb{X}}$ is the smallest subset of
${\mathbb{X}}$ invariant under $F^*$ and containing $A$, i.e.,
$F^*({\widehat{B}}) = {\widehat{B}}$ and
\[
{\widehat{B}} = \bigcap_{F^*(D) = D} D.
\]
Moreover, if $B$ is the basin and $\widehat{B}$ the fast basin of $A$, then

\begin{enumerate}
\item $B \subseteq{\widehat{B}}$ if $A^{o}\neq\emptyset$,

\item ${\widehat{B}}^{o} = \emptyset$ if $A^{o} =\emptyset$.
\end{enumerate}
\end{prop}

\begin{proof}
Concerning the first statement, that $\widehat{B}$ is invariant follows from
the definition. Concerning the minimality statement, if $C$ is invariant under
$F^*$ and $A \subseteq C$, then
\[
{\widehat{B}} = \bigcup_{k=0}^{\infty} (F^*)^{k} (A) \subseteq\bigcup
_{k=0}^{\infty} C = C.
\]
Concerning the second statement, in the case that $A^{o} \neq\emptyset$, there
is a $\omega\in[N]^{\infty}$ such that $\pi(\omega) \in A^{o}$. If $x \in B$,
then $\lim_{k\rightarrow\infty} f_{\omega|k}(x) = \pi(\omega) \in A^{o}$.
Therefore $f_{\omega|K} \in A^{o}$ for some $K$. In the case that $A^{o}
=\emptyset$, since ${\widehat{B}} =\bigcup_{k=0}^{\infty} (F^*)^{k} (A)$ is
a countable union of nowhere dense sets, then so is $\widehat{B}$ by the Baire
catgegory theorem.
\end{proof}

\begin{definition}
Let $F$ be an invertible IFS with point-fibred attractor $A$. Let $\theta
\in[N]^{\infty}$, and recall the notation $B(\theta) :=\bigcup_{k=1}^{\infty
}(f^{-1})_{\theta|{k}}(A) $. For a section $\tau\, : \, A \rightarrow
[N]^{\infty}$ define the \textbf{extended section}
\[
{\widehat{\tau} _{\theta}}\, : B(\theta) \rightarrow\Omega
\]
as follows. For $x \in B(\theta)$, let $k$ be the least integer such that $x
\in(f^{-1})_{\theta|k}(A)$. Therefore there is a $y\in A$ such that $x =
(f^{-1})_{\theta|k}(y)$. Define
\[
{\widehat{\tau}_{\theta}}(x) := \theta|k\bullet\tau(y).
\]
Then clearly $\widehat{\pi}\circ\widehat{\tau}_{\theta}$ is the identity.
\end{definition}

\medskip

Now consider two IFSs $F$ and $G$ on the same complete metric space, with an
equal number of functions, with point-fibred attractors $A_{F}$ and $A_{G}$,
with coordinate maps $\pi_{F}$ and $\pi_{G}$, and with sections $\tau_{F}$ and
$\tau_{G}$. Each word $\theta\in[N]^{\infty}$ induces two extended sections
$\widehat{\tau_{F}} := \widehat{\tau_{F,\theta}}$ and $\widehat{\tau_{G}}
:=\widehat{\tau_{G,\theta}}$. The maps
\[
\widehat{\pi}_{F} \circ\widehat{\tau_{G}} \, : B_{G}(\theta) \rightarrow
B_{F}(\theta) \qquad\text{and} \qquad\widehat{\pi}_{G} \circ\widehat{\tau_{F}}
\, : B_{F}(\theta) \rightarrow B_{G}(\theta)
\]
are called \textbf{extended fractal transformations} of the fractal
transformations
\[
\pi_{F} \circ\tau_{G} \, : \, A_{G} \rightarrow A_{F} \qquad\text{and}
\qquad\pi_{G} \circ\tau_{F} \, : \, A_{F} \rightarrow A_{G},
\]
respectively. In particular, if $\theta$ is full for both $F$ and $G$, then an
extended fractal transformation is defined on the basin of an attractor. If,
in addition, $F$ and $G$ are contractive, then the extended fractal
transformations take the whole space ${\mathbb{X}}$ to itself.

\begin{example}
\label{ex:FO2} This is a continuation of Example~\ref{ex:FO1} in
Section~\ref{sec:non-overlapping}. Consider two points $E_{1}, E_{2}$ in the
unit square $\blacksquare$ and the corresponding IFSs $F_{E_{1}}$ and
$F_{E_{2}}$. As described above, each word $\theta\in\{1,2,3,4\}^{\infty}$
generates a fractal transformation, often of $\mathbb{R}^{2}$ to itself. The
fractal transformation, in this case, is a fractal homeomophism \cite{Ba}. The
tiling in Example~\ref{ex:tri} can, in a similar fashion, generate an infinite
family of fractal homeomorphisms of the plane.
\end{example}

 A transformation that takes, for example, the unit square $\blacksquare$ to itself, can be visualized by its action on an image.   Define an {\it image} as a function $c \, : \, \blacksquare \rightarrow {\mathcal C}$, where $\mathcal C$ denotes the color palate, for example  ${\mathcal C} = \{ 0,1,2,\dots, 255\}^3$.  If $h$ is any transformation from $\blacksquare$ onto $\blacksquare$, define the {\it transformed image} $h(c) \, : \, \square \rightarrow {\mathcal C}$ by $h(c) := c \circ h.$   

\begin{example}  Consider two "fold-out" affine IFSs, $F$ and $G$, which, up to conjugation by an affine transformation, are of the form in Example~\ref{ex:FO1}  for two different values of the parameter $E$. The IFS $F$ corresponds to $E=(0.5,0.5)$ and the IFS $G$ corresponds to $E=(0.4,0.45)$. The attractor of $F$ is a small rectangle $L$ located in the bottom left corner of the left hand image. The attractor of $G$ is a small rectangle R situated in the bottom left corner of the right hand image. The fractal homeomorphism generated by $F$ and $G$ maps the portion of the left hand image lying over $L$ to the portion of the right hand image that lies over $R$. The word $\theta=\overline{1}$ extends the fractal homeomorphism to the upper left quadrant of the plane; the right hand image illustrates the result of applying this extended homeomorphism to the left hand image. 
\end{example}

\begin{figure}[tbh]
\centering
\fbox{\includegraphics[width=5in, keepaspectratio]{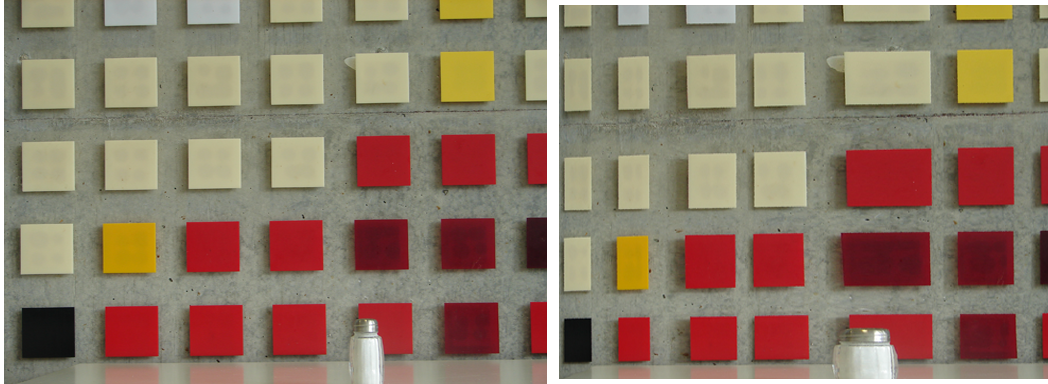}}
\vskip 3mm\caption{{}Example of an extended fractal homeomorphism, generated
by a pair of affine IFSs each with a rectangular just-touching attractor.}%
\label{fig:FT}%
\end{figure}

\end{document}